\documentclass[12pt]{article}

\usepackage{amsmath}
\usepackage{amssymb}
\usepackage{amsthm}
\usepackage{latexsym}
\usepackage{color}
\usepackage{graphicx}
\usepackage{color}

\DeclareSymbolFont{calletters}{OMS}{cmsy}{m}{n}
\DeclareSymbolFontAlphabet{\mathcal}{calletters}

\def\be{\begin{eqnarray}}
\def\ee{\end{eqnarray}}

\makeatletter \@addtoreset{equation}{section}

\usepackage{color}

\def \D{\mathbb{D}}
\def \E{\mathbb{E}}
\def \F{\mathbb{F}}
\def \L{\mathbb{L}}
\def \P{\mathbb{P}}

\def \R{\mathbb{R}}
\def \S{\mathbb{S}}

\def \T{\mathbb{T}}

\def\Bc{{\cal B}}

\def\Dc{{\cal D}}
\def\Ec{{\cal E}}
\def\Fc{{\cal F}}

\def\Kc{{\cal K}}
\def\Lc{{\cal L}}

\def\Pc{{\cal P}}

\def\Uc{{\cal U}}

\def \al{\underline{\alpha}}

\def \a{\alpha}

\def \Om{\Omega}
\def \om{\omega}

\def \omb{\overline{\om}}

\def \eps{\varepsilon}
\def \xb{\mathbf{x}}

\def \0{\mathbf{0}}

\def \Xh{\widehat{X}}

\newcommand{\we}{\wedge}
\newcommand{\ol}{\overline}
\newcommand{\ul}{\underline}

\newcommand{\ba}{\begin{array}}
\newcommand{\ea}{\end{array}}
\newcommand{\bea}{\begin{eqnarray}}
\newcommand{\eea}{\end{eqnarray}}
\newcommand{\beaa}{\begin{eqnarray*}}
\newcommand{\eeaa}{\end{eqnarray*}}

\def\dbE{\mathbb{E}}
\def\dbF{\mathbb{F}}

\def\dbL{\mathbb{L}}

\def\dbN{\mathbb{N}}
\def\dbP{\mathbb{P}}
\def\dbR{\mathbb{R}}
\def\dbS{\mathbb{S}}
\def\dbT{\mathbb{T}}
\def\dbQ{\mathbb{Q}}

\def\a{\alpha}
\def\b{\beta}
\def\g{\gamma}
\def\d{\delta}
\def\e{\varepsilon}

\def\si{\sigma}
\def\t{\tau}
\def\f{\varphi}
\def\th{\theta}
\def\o{\omega}

\def\D{\Delta}
\def\Th{\Theta}
\def\L{\Lambda}

\def\F{\Phi}
\def\O{\Omega}

\def\fO{{\mathsf \Omega}}
\def\fF{{\mathsf F}}
\def\fP{{\mathsf P}}

\def\cA{{\cal A}}
\def\cB{{\cal B}}

\def\cE{{\cal E}}
\def\cF{{\cal F}}

\def\cJ{{\cal J}}
\def\cK{{\cal K}}
\def\cL{{\cal L}}

\def\cP{{\cal P}}

\def\cT{{\cal T}}

\def\ch{\textsc{h}}

\def\no{\noindent}

\def\ms{\medskip}

\def\q{\quad}

\def\pa{\partial}
\def\cd{\cdot}
\def\cds{\cdots}

\def\qed{ \hfill \vrule width.25cm height.25cm depth0cm\smallskip}

\newcommand{\basa}{\begin{assumption}}
\newcommand{\easa}{\end{assumption}}

\newcommand{\bas}{\begin{assum}}
\newcommand{\eas}{\end{assum}}

\def\limsup{\mathop{\overline{\rm lim}}}
\def\liminf{\mathop{\underline{\rm lim}}}

\def\pa{\partial}

 \def\cd{\cdot}
\def\cds{\cdots}

\def\dis{\displaystyle}

\def\1{{\bf 1}}

\def\:{\!:\!}

\def \proof{{\noindent \bf Proof\quad}}

\newtheorem{thm}{Theorem}[section]
\newtheorem{lem}[thm]{Lemma}

\newtheorem{prop}[thm]{Proposition}
\newtheorem{rem}[thm]{Remark}
\newtheorem{eg}[thm]{Example}
\newtheorem{defn}[thm]{Definition}
\newtheorem{assum}[thm]{Assumption}

\newcommand{\rmi}{{\rm (i)$\>\>$}}

\newcommand{\rmii}{{\rm (ii)$\>\>$}}
\newcommand{\rmiii}{{\rm (iii)$\>\>$}}
\newcommand{\rmiv}{{\rm (iv)$\>\>$}}

\def\x{\times}
\def\ox{\otimes}
\def\1{{\bf 1}}
\def \proof{{\noindent \bf Proof. }}

\def\no{\noindent}

\title{On the convergence of monotone schemes for path-dependent PDEs
\footnote{The authors gratefully acknowledge the financial support of the ERC 321111 Rofirm, the ANR Isotace, the Chairs Financial Risks (Risk Foundation, sponsored by Soci\'et\'e G\'en\'erale), Finance and Sustainable Development (IEF sponsored by EDF and CA), and the grant of the region of l'\^ile de France.
The authors are also grateful to Ibrahim Ekren and Nizar Touzi for their helpful remarks and suggestions.
}}

\author{
	Zhenjie Ren\thanks{Ecole Polytechnique Paris, Centre de Math\'ematiques Appliqu\'ees,
        ren@cmap.polytechnique.fr}
        \and Xiaolu Tan\thanks{CEREMADE, University of Paris-Dauphine, tan@ceremade.dauphine.fr}
	}

\begin{document}

\maketitle

\abstract{

	We propose a reformulation of the convergence theorem of monotone numerical schemes introduced by Zhang and Zhuo \cite{ZhangZhuo} for viscosity solutions to path-dependent PDEs (PPDE), which extends the seminal work of Barles and Souganidis \cite{BarlesSouganidis} on the viscosity solution to PDE. We prove the convergence theorem under conditions similar to those of the classical theorem in \cite{BarlesSouganidis}. These conditions are satisfied, to the best of our knowledge, by all classical monotone numerical schemes in the context of stochastic control theory. In particular, the paper provides a unified approach to prove the convergence of numerical schemes for non-Markovian stochastic control problems, second order BSDEs, stochastic differential games etc.

	\vspace{1mm}


}

\section{Introduction}

In their seminal work \cite{BarlesSouganidis}, Barles and Souganidis proved a convergence theorem for the monotone numerical schemes for viscosity solutions to fully nonlinear PDEs. Assuming that a strong comparison principle holds true for viscosity sub- and super-solutions of a PDE, they show that for all numerical schemes satisfying the three properties, ``monotonicity'', ``consistency'' and ``stability", the numerical solutions converge locally uniformly to the unique viscosity solution of the PDE as the discretization parameters converge to zero. They mainly use the stability of viscosity solutions to PDEs and the local compactness of the state space. Due to their result, one only needs to check some local properties of a numerical scheme in order to get a global convergence result. Also, their result and method are widely used in the numerical analysis of viscosity solutions to PDEs.

It is well known that, by the Feynman-Kac formula, the conditional expectation of a random variable can be characterized by a viscosity solution of the corresponding parabolic linear PDE. This relationship has been generalized by the theory of backward SDE (corresponding to semi-linear PDE) and that of second-order backward SDE (corresponding to fully nonlinear PDE). However, these probabilistic tools have their PDE counterparts only in the Markovian case. Recently, a notion of viscosity solutions to path-dependent PDE (PPDE) was introduced by \cite{EKTZ}, which permits to study non-Markovian problems. In particular, it provides a unified approach for many Markovian, or non-Markovian stochastic dynamic problems, e.g. stochastic control problems, stochastic differential games, etc.

It would be interesting to extend the convergence theorem of Barles and Souganidis \cite{BarlesSouganidis} in the context of PPDE. The main obstacle for a direct extension of their arguments is that the state space is no longer locally compact. Zhang and Zhuo \cite{ZhangZhuo} provided recently a formulation of the convergence theorem of monotone schemes for PPDEs. They mainly use the stability of the viscosity solution to PPDE, and overcome the difficulty of non-local compactness by an optimal stopping argument as in the wellposedness theory of PPDE. They also provide an illustrative numerical scheme which satisfies all the conditions of their convergence theorem. However, this illustrative numerical scheme is not applicable in the general case. Moreover, most of the monotone numerical schemes in the sense of Barles and Souganidis \cite{BarlesSouganidis}, for example the finite difference scheme, do not satisfy their conditions.

Our main objective is to provide a new formulation of the convergence theorem for numerical schemes of PPDE.  Our conditions are slightly stronger than the classical conditions of Barles and Souganidis \cite{BarlesSouganidis}, as PPDEs degenerate to be PDEs. Nevertheless, to the best of our knowledge these conditions are satisfied by all classical monotone numerical schemes in the optimal control context, including the classical finite difference scheme, the Monte-Carlo scheme of Fahim, Touzi and Warin \cite{FTW}, the semi-Lagrangian scheme, the trinomial tree scheme of Guo, Zhang and Zhuo \cite{GuoZhangZhuo}, the switching system scheme of Kharroubi, Langren\'e and Pham \cite{KharroubiLangrenePham}, etc. Therefore, our result extends all these numerical schemes to the path-dependent case. In particular, it provides numerical schemes for non-Markovian second order BSDEs, and stochastic differential games, which is new in the literature, see also Possama\"i and Tan \cite{PossamaiTan}.

Similar to \cite{ZhangZhuo}, we use an optimal stopping argument to overcome the difficulty of non-local compactness. Instead of looking into an optimal stopping problem of a controlled diffusion as in \cite{ZhangZhuo}, we consider a discrete time optimal stopping problem of a controlled process. Therefore, our argument is quite different from that in \cite{ZhangZhuo}.

	The paper is organized as follows.
	In Section \ref{sec:Prelim} we provide some preliminary notations used in the paper. In Section \ref{sec:main_result} we recall the definition of viscosity solution to the path-dependent PDE,
	and present our main result, that is, a convergence theorem of monotone schemes for PPDEs. Further we compare with the result of Guo, Zhang and Zhuo \cite{GuoZhangZhuo} and that of Barles and Souganidis \cite{BarlesSouganidis}.
	In Section \ref{sec:schema} we review some classical monotone schemes for PDEs, and verify that they satisfy the technical conditions of our main convergence theorem, and thus can be applied in the PPDE context.
	Finally, we complete the proof of the main theorem in Section \ref{sec:proofs}.

\section{Preliminaries}
\label{sec:Prelim}

Throughout this paper let $T>0$ be a given finite maturity, $\O:=\{\o\in C([0,T];\dbR^d):\o_0=0\}$ the set of continuous paths starting from the origin, and $\Theta:=[0,T]\times\O$. We denote by $B$ the canonical process on $\O$, $\dbF = \{\cF_t, 0\le t\le T\}$ the canonical filtration,  $\cT$ the set of all $\dbF$-stopping times taking values in $[0,T]$, and $\dbP_0$ the Wiener measure on $\O$. Moreover, let $\cT^+$ denote the subset of $\t\in\cT$  taking values in $(0,T]$, and for $\ch \in \cT$, let $\cT_\ch$ and $\cT_\ch^+$ be the subset of $\t\in \cT$ taking values in $[0, \ch]$ and $(0, \ch]$, respectively.

Following Dupire \cite{Dupire}, we introduce the following pseudo-distance on $\Theta$:
for all $ (t,\o), (t',\o')\in\Theta$,
 \beaa
\|\o\| ~:=~ \sup_{0\le s\le T} |\o_s|,
&&  d\big((t,\o),(t',\o')\big)
 ~:=~
 |t-t'|+\|\o_{t\wedge\cd}-\o'_{t'\wedge\cd}\|.
 \eeaa
Let $E$ be a metric space, we say a process $X: \Theta \to E$ is in $C^0(\Theta,E)$ whenever $(t,\om) \mapsto X_t(\om)$ is continuous.  $\dbL^0(\cF, E)$ and $\dbL^0(\dbF, E)$ denote the set of all $\cF$-measurable random variables and $\dbF$-progressively measurable processes, respectively. We remark that $C^0(\Theta,E)\subset \dbL^0(\dbF, E)$, and when $E=\dbR$, we shall omit it in these notations. We also denote by $\mbox{\rm BUC}(\Th)$ the set of all functions bounded and uniformly continuous with respect to $d$.

For any $A\in \cF_T$, $\xi \in \dbL^0(\cF_T, E)$, $X\in \dbL^0(\dbF,E)$, and  $(t,\o)\in \Th$, define respectively the shifted set, the shifted random variable and the shifted process by
	\beaa
		A^{t,\o} := \{\o'\in \O: \o\otimes_t \o' \in A\},\q  \xi^{t,\o}(\o')
		:=
		\xi(\o\otimes_t\o'),
		~~
		X^{t,\o}_s(\o')
 		:=
 		X(t+s,\o\otimes_t\o')
 	\eeaa
where $\o\otimes_t\o'$ is the concatenated path defined as
\beaa
 (\o\otimes_t\o')_s:=\o_s\1_{[0, t]}(s) +(\o_t+\o'_{s-t})\1_{(t, T]} (s),\q 0\le s\le T.&
 \eeaa
 Following the standard arguments with monotone class theorem, we have the following results.

\begin{lem}\label{lem:measurability F}
Let $(t,\o)\in \Th$ and $s\in [t,T]$. Then
 $A^{t,\o}\in \cF_{s-t}$ for all $A\in\cF_{s}$, $\xi^{t,\o} \in \dbL^0(\cF_{s-t}, E)$ for all $\xi\in \dbL^0(\cF_s, E)$, $X^{t,\o} \in \dbL^0(\dbF, E)$ for all  $X\in \dbL^0(\dbF, E)$, and $\t^{t,\o}-t\in \cT_{s-t}$ for all $\t\in\cT_s$.
\end{lem}

	Next, let us introduce the nonlinear expectation. As in \cite{ETZ}, we fix a constant $L>0$ throughout the paper,
	and denote by $\cP$ the collection of all continuous semimartingale measures $\dbP$ on $\O$
	whose drift and diffusion coefficients are bounded by $L$.
	More precisely, a probability measure $\dbP\in\cP$ if under $\P$,
	the canonical process $B$ is a semimartingale with natural decomposition $B=A^\dbP+M^\dbP$,
	where $A^{\P}$ is a process of finite variation,
	$M^{\P}$ is a continuous martingale with quadratic variation $\langle M^{\P} \rangle$,
	such that  $A^{\P}$ and $\langle M^{\P} \rangle$ are absolutely continuous in $t$, and
	\be\label{defn:PL}
		 \| \mu^\dbP\|_\infty, \|a^\dbP\|_\infty \le L,\q
		\mbox{where}~\mu^\dbP_t:=\frac{dA^{\P}_t}{dt},~a^\dbP_t:= \frac{d\langle M^{\P} \rangle_t}{dt},
		~~\P \mbox{-a.s.}
	\ee
We then define the nonlinear expectations:
\be\label{defn: cE}
\ol\cE[\cd] := \sup_{\dbP\in \cP}\dbE^\dbP[\cd]\q
\mbox{and}\q
\ul\cE[\cd] := \inf_{\dbP\in \cP}\dbE^\dbP[\cd].
\ee

\section{Convergence of monotone schemes for PPDE}
\label{sec:main_result}

\subsection{Definition of viscosity solution to PPDE}\label{subsec: defnPPDE}

	We introduce the following assumptions on a function $G: \Th\times \dbR\times \dbR^d\times \dbS^d \rightarrow\dbR$ ($\dbS^d$ is the set of all $d\times d$ symmetric matrices).

\begin{assum}\label{assum: G}
The function $G: (t,\o,y,z,\g) \mapsto G(t,\o,y,z,\g)$ satisfies that
\begin{itemize}
\item  $G$ is non-decreasing in $\g$;

\item $G$ is continuous in all arguments.

\end{itemize}
\end{assum}

\no	 Consider the  PPDE in the following form
	\be \label{eq:PPDE}
		-~\partial_t u(t,\om) ~-~ G \big(\cd, u, \partial_{\om} u, \partial^2_{\om \om} u \big)(t,\o)
		~=~
		0,\q
		\mbox{for all}~~(t,\o)\in [0,T)\times\O,
	\ee
	with the terminal condition $u(T, \cd) = \xi$. The path derivatives are defined as follows.

	\begin{defn}
We say that $u\in C^{1,2}_\cP(\Theta)$ if
 $u\in C^0(\Theta)$ and there exist processes $\L, Z,\Gamma\in C^0(\Theta)$ valued in $\dbR$, $\dbR^d$ and $\dbS_d$, respectively, such that:
\beaa
 d u_t
 &=&
 \L_t dt + \frac12 \Gamma_t:d\langle B\rangle_t + Z_t dB_t,
 ~~\cP-\mbox{a.s. for all}~\dbP\in\cP.
 \eeaa
 The processes $\L$, $Z$ and $\Gamma$ are called the time derivative, spacial gradient and spatial Hessian, respectively, and we denote $\pa_t u := \L$, $\partial_\omega u := Z$, $\partial^2_{\omega\omega} u :=\Gamma$.
	\end{defn}

\no In \cite{EKTZ} and the following works \cite{ETZ, ETZ2}, the authors introduced a notion of viscosity solutions to PPDEs, by using the test functions in $C^{1,2}_\cP(\Th)$. 
Indeed, as observed in \cite{RTZ-survey}, one does not need to consider all $C^{1,2}_\cP(\Th)$ test functions but only the paraboloids. For $\a\in\dbR, \b\in\dbR^d, \g\in\dbS^d$, we define the paraboloid on the real space:
	\beaa
		\phi^{\a,\b,\g}(t,x):=\a t+ \b\cd x + \frac12 \g:(x x^T)
		& \mbox{for all} & (t,x)\in \dbR^+\times\dbR^d,
	\eeaa
where $A_1:A_2 := {\rm Tr}[A_1 A_2]$. We also abuse a little the notation to define the corresponding `paraboloid' on the path space:
\beaa
\phi^{\a,\b,\g}(t,\o) ~:= ~ \phi^{\a,\b,\g}(t,\o_t) , \q\mbox{for all}\q (t,\o) \in \Th.
\eeaa
Note that the path derivatives at the point $(0,0)$ of the paraboloid are clearly:
\beaa
\pa_t \phi^{\a,\b, \g}_0 = \a,
\q \pa_\o \phi^{\a,\b, \g}_0 =\b, \q\mbox{and}\q
\pa^2_{\o\o} \phi^{\a,\b,\g}_0 = \g.
\eeaa

 The sub- and super-jets of a function $u\in \mbox{\rm BUC}(\Th)$ at $(t,\o)\in [0,T)\times\O$ are as follows
\beaa
& \ul\cJ u(t,\o):=\Big\{(\a,\b,\g): u(t,\o)=\max_{\t\in\cT_{\ch_\d}}\ol\cE[u^{t,\o}_\t - \phi^{\a,\b,\g}_\t],~\mbox{for some}~\d>0\Big\}, &\\
& \ol\cJ u(t,\o):=\Big\{(\a,\b,\g): u(t,\o)=\min_{\t\in\cT_{\ch_\d}}\ul\cE[u^{t,\o}_\t - \phi^{\a,\b,\g}_\t],~\mbox{for some}~\d>0\Big\}, &
\eeaa
where $\ch_\d(\o'):=\d\we \inf\{s\ge 0:|\o'_s|\ge \d\}\in \cT^+$.

\begin{defn}
Let $u\in \mbox{\rm BUC}(\Th)$.

\no {\rm (i)}\q $u$ is a $\cP$-viscosity subsolution (resp. supersolution) of the path dependent PDE \eqref{eq:PPDE}, if at any point $(t,\o)\in [0,T)\times\O$ it holds for all $(\a,\b,\g)\in \ul\cJ u(t,\o)$ (resp. $\ol\cJ u(t,\o)$) that
\beaa
-\a - G(t,\o, u(t,\o),\b,\g)~\le \mbox{(resp. $\ge$)}~0.
\eeaa

\no {\rm (ii)}\q $u$ is a $\cP$-viscosity solution of the path dependent PDE \eqref{eq:PPDE}, if $u$ is both a $\cP$-viscosity subsolution and a $\cP$-viscosity supersolution of \eqref{eq:PPDE}.
\end{defn}

\no For the arguments below, we need to introduce an equivalent definition using constant localization and test functions in $C^{1,2}_0(\dbR^+\times\dbR^d)$, i.e. the class of all $C^{1,2}$ scalar functions $\f$ of which the partial derivatives $\pa_t \f, \pa_x\f, \pa^2_{xx} \f$ are of compact support. Consider the set of test functions:
$$
\ul\cA u(t,\o):=\Big\{\f\in C^{1,2}_0(\dbR^+\times\dbR^d):(u^{t,\o}-\f)_0 = \max_{\t\in \cT_\d} \ol\cE\big[(u^{t,\o}-\f)_\t\big],~\mbox{for some}~\d>0 \Big\},
$$
$$
\ol\cA u(t,\o):=\Big\{\f\in C^{1,2}_0(\dbR^+\times\dbR^d):(u^{t,\o}-\f)_0 = \min_{\t\in \cT_\d} \ul\cE\big[(u^{t,\o}-\f)_\t\big],~\mbox{for some}~\d>0 \Big\},
$$
where we abuse the notation again, by denoting $\f_t = \f(t,B_t)$.
\begin{prop}\label{prop:eqv_defn}
 A function $u$ is a $\cP$-viscosity subsolution (resp. supersolution) of Equation \eqref{eq:PPDE}, if and only if at any point $(t,\o)\in [0,T)\times\O$ it holds for all $\f\in \ul\cA u(t,\o)$ (resp. $\ol\cA u(t,\o)$) that
\be \label{eq:Lc0}
\cL^{t,\o} \f_0 :=-\pa_t\f_0 - G(t,\o,u(t,\o),\pa_x \f_0,\pa^2_{xx} \f_0)~\le \mbox{(resp. $\ge$)}~0.
\ee
\end{prop}

\no We will report the proof of the above proposition in Section \ref{sec:proofs}. 

\ms 

One of the motivations of the PPDE theory is to characterize the value functions of non-Markovian stochastic control problems.

	\begin{eg}{\rm
		Consider a second-order backward SDE (see e.g. Cheridito, Soner, Touzi and Victoir \cite{CheriditoSonerTouziVictoir}, Soner, Touzi and Zhang \cite{Soner_Touzi_Zhang1}) with generator $F: \Theta \x \R \x \R^d \x K \to \R$ and the controlled process with diffusion coefficient $\sigma: \Theta \x K \to S_d$, where $K$ is some set in which the control processes take values. Then the solution of the second-order backward SDE corresponds to a viscosity solution to the PPDE with the nonlinearity:
		\be \label{eq:PPDE_2BSDE}
			G(t,\om, y ,z ,\gamma)
			&:=&
			\sup_{k \in K} \Big[ \frac{1}{2}\sigma^2(t,\om, k):\gamma + F(t, \om, y, \sigma(t,\om,k) z, k) \Big].
		\ee	
		We refer to Section 4 of Ekren, Touzi and Zhang \cite{ETZ} for more details.
		Another example is the application of PPDEs in the stochastic differential games (see e.g. Pham and Zhang \cite{PhamZhang}), where the nonlinearity of the PPDE is of the form:
		\be \label{eq:PPDE_DG}
			G(t,\om, y, z, \gamma)
			:=
			\sup_{k_1 \in K_1} \inf_{k_2 \in K_2}
			\Big[
				\frac{1}{2} \sigma^2(t,k_1, k_2): \gamma + F(t,\om, y ,\sigma(t, k_1, k_2) z, k_1, k_2)
			\Big].~~
		\ee
		}
	\end{eg}

\subsection{Main results}

		For parabolic PDEs with terminal conditions of the form \eqref{eq:PPDE}, the numerical schemes are usually given as a backward iteration. Let $u^h(T, \cdot) = \xi(\cdot)$, and then given the value of $u^h(t, \cdot)$, we compute the value of $u^h(t-h,\cdot)$.
		More precisely, we will introduce the numerical scheme as an operator $\dbT$:
		for each $(t, \om ) \in [0,T) \x \Om$ and $0 < h \le T-t$, $\dbT^{t,\o}_h$ is a function from $\dbL^0(\cF_{t+h})$ to $\dbR$, then the backward iteration is given by
	\beaa
		u^h(t,\o) &:= & \T^{t,\om}_h u^h_{t+h}.
	\eeaa
	As illustrated by Kushner and Dupuis \cite{KushnerDupuis}, in the context of Markovian stochastic control problem, the numerical result $u^h$ can be interpreted as the value function of a controlled Markov chain problem, or equivalently the numerical scheme $\T^{t,\om}_h$ can be treated as a discrete time nonlinear expectation.
	Moreover,  such an interpretation implies the monotonicity condition  in sense of Barles and Souganidis \cite{BarlesSouganidis}.
	Inspired by this observation, we would like to introduce a discrete time version of nonlinear expectations of $\ol\Ec$ and $\ul\Ec$, in preparation of our new formulation of the monotonicity condition in this general path-dependent context.

\begin{defn}\label{defn: KFU}
	Let $\{U_i,i\ge 1\}$ be a sequence of independent random variables defined on a probability space $(\tilde \O, \tilde\cF, \tilde \dbP)$, such that each $U_i$ follows the uniform distribution on $[0,1]$. Let $h > 0$, $K$ be a subset of a metric space, $\F_h: K\times [0,1]\rightarrow\dbR $ be a Borel measurable function such that for all $k \in K$ we have
 		\be \label{eq:CondH}
			| {\tilde \dbE}\big[ \F_h(k,U) \big] | ~\le~ Lh,
			~~ \mbox{\rm Var} \big[ \F_h(k, U) \big] \le L h
			~~\mbox{and}~~
			\tilde \dbE \big[ \F_h(k, U)^3 \big] \le L h^{3/2}.
		\ee
Denote the filtration $\tilde\dbF:= \{\tilde\cF_i, i\in\dbN\}$, where $\tilde\cF_n:=\si\{U_i, i\le n\}$.
   Let $\Kc = \{\nu: \nu_{ih} ~\mbox{is $\tilde\cF_i$-measurable and takes values in $K$  for all $i\in \dbN$} \}$.
    For all $\nu \in \Kc$, we define
	\be \label{eq:def_Xh}
	X^{h,\nu}_0 := 0,\q 	X^{h,\nu}_{ih} &=& X^{h,\nu}_{(i-1)h} ~+~ \F_h(\nu_{ih},U_i)~~\mbox{for}~~i\ge 1.
	\ee
Further, we denote by $\Xh^{h,\nu}:[0,T]\times\tilde\O\rightarrow \O$ the linear interpolation of the discrete process $\big\{X^{h,\nu}_{ih}, i\in\dbN\big\}$ such that
$\Xh^{h,\nu}_{ih} = X^{h,\nu}_{ih}$ for all $i$.
 Finally, for any function $\f \in \dbL^0(\cF)$, we define the nonlinear expectation:
\be\label{defn: cEh}
& \ul\cE_h [\f]
			~:=~
			\inf_{\nu\in \cK} \tilde \dbE \Big[ \f \big(\Xh^{h,\nu} \big) \Big]
			\q\mbox{and}\q
\ol\cE_h [\f]
            ~:=~
            \sup_{\nu\in\cK} \tilde \dbE \Big[ \f \big(\Xh^{h,\nu} \big) \Big]. &
\ee
\end{defn}
\ms

	Let us now formulate the conditions on the numerical scheme $\dbT$.

	\begin{assum} \label{ass:NumScheme}{\rm
	(i) \q  Consistency: for every $(t,\om) \in [0,T) \x \Om$ and $\varphi \in C_0^{1,2}(\dbR^+\times\dbR^d)$,
		\be\label{eq:newconsis}
			\lim_{(t', \om', h, c) \to (t, 0, 0, 0)}
			\frac{ \f(t', \o\otimes_t\om') + c - \T_h^{t', \om \otimes_t \om'} \big[ \f(t' + h, \cdot)+c \big]}
			{h}
			&=&
			\Lc^{t,\om} \f_0.
		\ee
		
	\no (ii) \q Monotonicity:
		there exists a nonlinear expectation $\ul\cE_h$ as in Definition \ref{defn: KFU} such that,
        for any $\varphi,\psi\in \dbL^0(\cF_{t+h})$ we have
        \be\label{monocon2}
        \T^{t,\om}_h [ \varphi] - \T^{t,\om}_h [ \psi] \ge -h \bar\d(h)
         \q\mbox{whenever}\q
        \inf_{0\le \alpha \le L} \ul\cE_h \big[ e^{\alpha h}(\f-\psi)^{t,\o}\big] \ge 0.
        \ee
	where $\bar\d:\dbR^+\rightarrow\dbR$ is a function such that $\lim_{h\downarrow 0} \bar\d(h)=0$.
\ms	
	
		\noindent (iii)\q  Stability: $u^h$ are uniformly bounded and uniformly continuous in $(t,\om)$, uniformly in $h$.
		}
	\end{assum}

\begin{rem}{\rm
	In practice, we can verify that the numerical scheme $\T$ satisfies another version of monotonicity condition:
		\be\label{eq: monotoPPDE}
		     \T^{t,\om}_h [ \varphi] ~-~ \T^{t,\om}_h [ \psi]
			~\ge~
			\inf_{0\le \alpha \le L} \ul\cE_h \big[ e^{\alpha h}(\f-\psi)^{t,\o}\big]
			~-~
			h \bar\d(h),
		\ee
	which implies indeed the monotonicity condition \eqref{monocon2}.
	Moreover, under condition \eqref{eq: monotoPPDE}, it is possible to weaken slightly the consistency condition \eqref{eq:newconsis}, with almost the same proof, to
		\beaa
			\lim_{(t', \om', h) \to (t, 0, 0)}
			\frac{ \f(t', \o\otimes_t\om')  - \T_h^{t', \om \otimes_t \om'} \big[ \f(t' + h, \cdot) \big]}
			{h}
			&=&
			\Lc^{t,\om} \f_0.
		\eeaa
}
\end{rem}

\no 	Our main result is the following convergence of the monotone scheme for PPDE \eqref{eq:PPDE}.
	
	\begin{thm}\label{thm:main}
	Assume that 
	\begin{itemize}
    \item the nonlinearity $G$ of PPDE \eqref{eq:PPDE} satisfies Assumption \ref{assum: G}, and the terminal condition $\xi$ is continuous;
	
	\item 	the numerical scheme $\T$ satisfies Assumption \ref{ass:NumScheme};
			
	\item the comparison of the viscosity sub- and super-solutions of PPDE \eqref{eq:PPDE} holds true, i.e. if $u,v\in \mbox{\rm BUC}(\Th)$ are $\cP$-viscosity subsolution and supersolution to PPDE \eqref{eq:PPDE}, respectively, and $u(T,\cd)\le v(T,\cd)$, then $u\le v$ on $\Th$.
\end{itemize}	 
	 Then PPDE \eqref{eq:PPDE} admits a unique bounded $\cP$-viscosity solution $u$,
		and
		\be
			u^h ~\rightarrow~ u
			&\mbox{locally uniformly,}& \mbox{as}~~
			h \rightarrow 0.
		\ee
	\end{thm}

\no Here are some remarks on our main result. 

%
%

\begin{rem}{\rm
In the case of semilinear PPDEs, i.e. the nonlinearity $G(\th,y,z,\g) = \frac12 \g +F(\th,y,z)$, a comparison result is proved in Ren, Touzi and Zhang \cite{RTZ}  under quite general assumptions.
As for viscosity solutions of fully nonlinear PPDEs, a comparison result was first proved in Ekren, Touzi and Zhang \cite{ETZ2} for PPDE \eqref{eq:PPDE} under certain conditions. A recent improvement can be found in \cite{RTZ-fn}. 
}
\end{rem}

	\begin{rem}[Comparison with Zhang and Zhuo \cite{ZhangZhuo}]{\rm
		Let us compare our Assumption \ref{ass:NumScheme} with that in \cite{ZhangZhuo}. Our condition (i) is weaker and thus easier to verify comparing to that in \cite{ZhangZhuo}. The essential difference is between our condition \rmii and theirs. Our condition (ii), although stated in a complicated way, is satisfied by all (to the best of our knowledge) classical monotone schemes in PDE context.
	 We also emphasize that in \cite{ZhangZhuo}, a convergence rate has been obtained under additional smoothness conditions on the solution of PPDE.
		Without smoothness conditions, a convergence rate has been obtained for HJB PDEs using Krylov's shaking coefficient technique.
		However, it seems not trivial to extend this shaking coefficient technique to the path-dependent case.
		Nevertheless, for a class of path-dependent stochastic control problems, a convergence rate has been obtained in \cite{Dolinsky} and \cite{TanCtrl}, using strong invariance principle techniques.
	}
	\end{rem}

\paragraph{Comparison with Barles and Souganidis's theorem}

	When a PPDE degenerates to be a PDE:
	\be\label{eq:PDE}
		\mathsf{L} u(t,x) ~:=~ - \partial_t u(t,x) - G_0 (\cd, u, \pa_x u, \pa^2_{xx} u)(t,x) ~=~ 0,
		&\mbox{on}& [0,T)\times\dbR^d,
	\ee
	with the terminal condition $u(T,\cd)=g$.
	Note that the definition of viscosity solution to PDE is slightly different from that to PPDE recalled in Section \ref{subsec: defnPPDE},
	but under general conditions a viscosity solution to PDE \eqref{eq:PDE} is a viscosity solution of the corresponding PPDE.

	\begin{assum} \label{assum:PDE}
		\rmi The terminal condition $g$ is bounded continuous.
		
		\noindent \rmii The function $G_0$ is continuous and $G_0(t,x,y,z,\g)$ is non-decreasing in $\g$.
		
		\noindent \rmiii PDE \eqref{eq:PDE} admits a comparison principle for bounded viscosity solution,
		i.e. if $u,v$ are bounded viscosity subsolution and supersolution to PDE \eqref{eq:PDE}, respectively, and $u(T,\cd)\le v(T,\cd)$, then $u\le v$ on $[0,T]\times\dbR^d$.	
	\end{assum}

	For any $t\in[t,T)$ and $h \in (0, T-t]$, let $\dbT^{t,x}_h$ be an operator on the set of bounded functions defined on $\R^d$.
    For $n\ge 1$, denote $h:=\frac{T}{n}<T-t$, $t_i=ih$, $i=0,1,\cds,n$,
	let the numerical solution be defined by
	\beaa
		u^h(T,x):=g(x),\q u^h(t,x):=\dbT^{t,x}_h [u^h(t+h, \cdot)],\q t\in [0,T),~~i=n,\cds,1.
	\eeaa

	\begin{assum} \label{assum:SchemaPDE}
		\rmi Consistency: for any $(t,x)\in [0,T)\times\dbR^d$ and any smooth function
		$\f\in C^{1,2}([0,T)\times\dbR^d)$,
		\beaa
			\lim_{(t',x',h,c)\rightarrow (t,x,0,0)}
			\frac{(c+\f)(t',x')-\dbT^{t',x'}_h\big[(c+\f)(t'+h,\cd)\big]}{h}
			~=~
			\mathsf{L}\f(t,x).
		\eeaa
		
		\noindent \rmii Monotonicity: $\dbT^{t,x}_h[\f] \le \dbT^{t,x}_h[\psi]$ whenever $\f\le \psi$.
		
		\noindent \rmiii Stability: $u^h$ is bounded uniformly in $h$ whenever $g$ is bounded.
	
		\noindent \rmiv Boundary condition: $\lim_{(t',x',h)\rightarrow (T,x,0)}u^h(t',x')=g(x)$ for any $x\in\dbR^d$.
	\end{assum}

\no	We now recall the convergence theorem of the monotone scheme,
	deduced from Barles and Souganidis \cite{BarlesSouganidis} in this context of the parabolic PDE \eqref{eq:PDE}.

	\begin{thm}
		Let the generator function $G_0$ in \eqref{eq:PDE} and the terminal condition $g$ satisfy Assumption \ref{assum:PDE}, and the numerical scheme $\dbT^{t,x}_h$ satisfy Assumption \ref{assum:SchemaPDE}.
		Then the parabolic PDE \eqref{eq:PDE} has a unique bounded viscosity solution and $u^h$ converges to $u$ locally uniformly as $h\rightarrow 0$.
	\end{thm}

	\vspace{2mm}
	
	\begin{rem}[Comparison with Assumption \ref{ass:NumScheme}]{\rm
		(1)\q  Assumption \ref{ass:NumScheme} (i) is the same as Assumption \ref{assum:SchemaPDE} (i) when the PPDE degenerates to a PDE.

		\vspace{1mm}

		\no (2) \q  Assumption \ref{ass:NumScheme} \rmii is stronger than Assumption \ref{assum:SchemaPDE} $\mathrm{(ii)}$. However, as we mentioned at the beginning of this section,	Kushner and Dupuis \cite{KushnerDupuis} studied the classical finite-difference scheme,
		and highlighted that the monotonicity condition is in fact equivalent to a controlled Markov chain interpretation, where the increments of the Markov chain satisfy \eqref{eq:CondH}.
		Our formulation of the monotonicity in Assumption \ref{ass:NumScheme} \rmii matches this observation.
		Moreover, the new monotonicity condition is satisfied by all classical monotone scheme, to the best of our knowledge, in the context of stochastic control theory.
		See our review in Section \ref{sec:schema}.

		\vspace{1mm}

		\no (3)\q The stability condition in Assumption \ref{ass:NumScheme} \rmiii is also stronger than Assumption \ref{assum:SchemaPDE} $\mathrm{(iii)}$.
		Nevertheless, in the classical numerical analysis for parabolic PDE \eqref{eq:PDE}, in order to check Assumption \ref{assum:SchemaPDE} $\mathrm{(iv)}$,
		one needs (explicitly or implicitly) to prove a uniform continuity property of numerical solutions uniformly on the discretization parameter, which leads to the same condition as in Assumption \ref{ass:NumScheme} $\mathrm{(iii)}$.
		See also our review in Section \ref{sec:schema}.
		}
	\end{rem}

\section{Examples of monotone schemes}
\label{sec:schema}

	We discuss here some classical monotone numerical schemes in the stochastic control context,
	and provide some sufficient conditions  Assumption \ref{ass:NumScheme} to hold true.
    Let us first add some assumptions on the functions $G$ and $\xi$ for PPDE \eqref{eq:PPDE}.

	\begin{assum} \label{ass:PPDE}
		The terminal condition $\xi$ is Lipschitz in $\om$,  $G$ is increasing in $\g$,
		and $G$ is Lipschitz in $(y, z, \gamma)$:
		i.e. there is some constant $C$ such that for all $(t, \om) \in \Theta$ and $(y, z, \gamma), (y', z', \gamma') \in \R \x \R^d \x S_d$,
		\beaa
			\Big|G(t, \om, y, z, \gamma) -  G(t, \om, y, z', \gamma') \Big|
			&\le&
			C \Big(  \big|y-y' \big| ~+~ \big|z-z' \big| ~+~ \big| \gamma - \gamma' \big|
			\Big).
		\eeaa		
	\end{assum}

	In this section, we denote $t_k := h k$ for $h = \Delta t > 0$.
    Given $\xb = (\xb_{t_0}, \xb_{t_1},\cdots, \xb_{t_k})$ a sequence of points in $\dbR^d$, we denote by $\widehat \xb \in \O$ the linear interpolation of $\xb$ such that $\widehat \xb_{t_i} = \xb_{t_i} $ for all $i$.
    Further, for $(t, \om) \in \Theta$, $h >0$ and $z \in \R^d$, we define a path
\beaa
    (\om \ox_t^h z) ~:=~ \om \ox_t z^h,
    ~~~~\mbox{where}~~
	z^h_s ~:=~ \left\{\ba{lll}
\dis \frac{s}{h}z,\q \mbox{for}~~0\le s\le h;\\
\dis z,\q\q  \mbox{for}~~s>h.
\ea\right.
\eeaa

    Let $E$ be some normed vector space, then for maps $\psi: \Theta \to E$,
    we introduce the norm $|\psi|_0$ and $|\psi|_1$ by
	\beaa
        |\psi|_0
        ~:=~
        \sup_{(t,\om) \in \Theta} |\psi(t,\om)|
        &\mbox{and}&
		|\psi|_1
		~:=~
		\sup_{(t,\om) \neq (t', \om')}
		\frac{| \psi(t, \om) - \psi(t', \om')|}
		{ |\om_{t \wedge \cdot} - \om'_{t' \wedge \cdot}| + |t-t'|^{1/2}}.
	\eeaa

\subsection{Finite difference scheme}
\label{subsec:schemeFD}

	For simplicity, we assume that the state space is of dimension one ($d= 1$).
	Let $\Delta x  > 0$ be the space discretization size.
    For every $(t,\om) \in \Theta$, $h >0$ and $\Fc_{t+h}$-measurable random variable $\psi: \Om \to \R$, we define the discrete derivatives 
    $$\Dc_h \psi(t, \om) := \big(\Dc_h^{0} \psi, \Dc_h^{1} \psi, \Dc_h^{2} \psi \big)(t,\om),$$
    where
	\beaa
	  \Dc^{0}_h \psi(t,\om) ~:=~ \psi(\om_{t \wedge \cdot}),
		\q
		\Dc^{1}_h \psi(t,\om) ~:=~ \frac{\psi( \om\ox_t^h \Delta x) - \psi(\om_{t \wedge \cdot})}{\Delta x}, \\
	\mbox{and}\q	\Dc^{2}_h \psi(t,\om) ~:=~
		\frac{\psi \big(\om \ox_t^h \Delta x \big) - 2 \psi( \om_{t \wedge \cdot}) + \psi \big(\om \ox_t^h (-\Delta x)\big)}
		{\Delta x^2}.
	\eeaa
	Then an explicit finite difference scheme is given by
	\be \label{scheme:FD}
		\T^{t,\o}_h [u^h_{t+h}] &:=& u^h(t+h, \om_{t \wedge \cdot}) + h G \big( t, \om, \Dc_h u^h_{t+h} (t,\om) \big).
	\ee

	\begin{prop} \label{Prop:SchemaDF}
		Suppose that Assumption \ref{ass:PPDE} holds true and
		$G$ is Lipschitz in $\om$, i.e. there is a constant $C$
		such that for all $\om, \om' \in \O$ and all $(t, y, z, \gamma) \in [0,T] \x \R \x \R^d \x S_d $,
		\beaa
			\Big|G(t, \om, y, z, \gamma) -  G(t, \om', y, z, \gamma) \Big|
			&\le&
			C \big\| \om_{t \wedge \cdot} - \om'_{t \wedge \cdot} \big\|.
		\eeaa
		Assume in addition the CFL (Courant-Friedrichs-Lewy) condition, i.e.
		\be \label{eq:CFL}
			\eps ~~\le~~
			\frac{h |\nabla_{\gamma} G |_0 }{\Delta x^2}
			&\le&
			\frac{1}{2} - \eps,
		\ee
		and that $\nabla_{\gamma} G \ge \eps $ for some small constant $\eps > 0$.
		Then Assumption \ref{ass:NumScheme} holds true for finite difference scheme \eqref{scheme:FD}.
		In particular, the numerical solution $u^h$ is $\frac{1}{2}$--H\"older in $t$ and Lipschitz in $\om$, uniformly on $h$.
	\end{prop}
	
	\proof We will check each condition in Assumption \ref{ass:NumScheme}.
	For the simplicity of presentation, we assume that $G$ is independent of $y$. Clearly, the argument still works if $G$ is Lipschitz in $y$.

    \ms
	\noindent \rmi The consistency condition (Assumption \ref{ass:NumScheme} (i)) is obviously satisfied by \eqref{scheme:FD} as in the non-path-dependent case.
	
    \ms
	\noindent \rmii For the monotonicity in Assumption \ref{ass:NumScheme} (ii), let us consider two different bounded functions $\varphi$ and $\psi$.
	Denote $\phi := \varphi - \psi$, then by direct computation,
	\beaa
		\T^{t,\o}_h[\varphi] - \T^{t,\o}_h[\psi]
		&=&
		\phi(\om^0)
		~+~
		h \Big( G_y \Dc^{0,t}_h \phi + G_z \Dc^{1,t}_h \phi + G_{\gamma} \Dc^{2,t}_h \phi \Big),
	\eeaa
    where $G_y$, $G_z$ and $G_{\gamma}$  depend on $(t,\om)$ and $(\varphi, \psi)$, but are uniformly bounded by the Lipschitz constant $L$ of $G$.
    Let $b \in [-L, L]$ and $ \eps \le a \le |\nabla_{\gamma} G|_0$ be two constants, and $\zeta^{a,b}$ be a random variable defined on a probability space $(\tilde\O,\tilde\cF,\tilde P)$ such that
	\beaa
	& \tilde \P(\zeta^{a,b} = 0) ~=~ 1- b \frac{h}{\Delta x} - 2a \frac{h}{\Delta x^2},& \\
	& \tilde \P(\zeta^{a,b} = \Delta x) ~=~ b \frac{h}{\Delta x}  + a \frac{h}{\Delta x^2}\q \mbox{and}\q 
		\tilde \P(\zeta^{a,b} = - \Delta x) ~=~ a \frac{h}{\Delta x^2}.&
	\eeaa
	The law of $\zeta^{a,b}$ is well defined for $h$ small enough, because every term above is positive and the sum of all terms equals to $1$ under condition \eqref{eq:CFL}. Further, we have
	\be \label{eq:Zeta_DF}
		\tilde \E \big[ \zeta^{a,b} \big] = b h,
		~~ \mathrm{Var} \big[ \zeta^{a,b} \big] = a h,
		&\mbox{and}&
		\tilde \E \big[ | \zeta^{a,b}|^3 \big] ~\le~ |\Delta x|^3 ~\le~ C h^{3/2},
	\ee
	where the last terms follows by  $\D x\approx h^\frac12$.

	Then let $F_h(a,b, \cdot) : \R \to [0,1]$ be the distribution function of $\zeta^{a,b}$ and $\F_h(a,b,\cdot) : [0,1] \to \R$ be the generalized inverse function of $F_h({a,b}, \cdot)$, i.e.
\be \label{eq:Phi_h_DF}
	\F_h(a,b,x) &:=& \inf\{y: F_h(a,b,y)> x\}.
\ee
	In view of \eqref{eq:Zeta_DF}, we may verify \eqref{eq: monotoPPDE}, and  the monotonicity condition of Assumption \ref{ass:NumScheme} \rmii follows.
	
\ms
	\noindent \rmiii To prove Assumption \ref{ass:NumScheme} $\mathrm{(iii)}$, we will prove that there is a constant $C$ independent of $h$ such that
	\be \label{eq:uh_unif_conti}
		\big| u^h(t, \om) - u^h(t', \om') \big|
		&\le&
		C \Big (  \| \om_{t \wedge \cdot} - \om'_{t' \wedge \cdot} \| + \sqrt{|t' - t|} \Big),
		~~\forall (t,\om), (t', \om') \in \Theta.~~
	\ee
	Let us first prove that $u^h$ is Lipschtiz in $\om$.
	Denote
    \beaa
        L^h_t  &:=& \sup_{(t,\om), (t', \om') \in \Theta}
        \frac{u^h(t,\o)-u^h(t,\o')}{\|\o_{t\we\cd}-\o'_{t\we\cd}\|}
        1_{\{\|\o_{t\we\cd}-\o'_{t\we\cd}\|>0\}}.
    \eeaa
    By direct computation, we have
	\begin{multline}\label{eq: backward induction}
		u^h(t, \om) - u^h(t, \om') 
		=
		h G_{\om}\|\o_{t\we\cd} - \o'_{t\we\cd}\|
		+ \tilde D^G_h u^h_{t+h}(t,\o)-\tilde D^G_h u^h_{t+h}(t,\o'),
	\end{multline}
    where 
\beaa
\tilde D^G_h u^h_{t+h} ~:=~\big( (1 + h G_y) \Dc^0_h + h G_z \Dc^1_h + h G_{\gamma} \Dc^2_h \big) u^h_{t+h},
\eeaa    
with   $G_y$, $G_z$ and $G_{\gamma}$ uniformly bounded by $L$.
	Then there is a constant $C$ independent of $h$ such that
	\beaa
		L^h_t &\le& (1 + Ch) L^h_{t+h} + C h.
	\eeaa
	Notice that the terminal condition $\xi$ is Lipschitz,
    it follows by the discrete Gronwall inequality, we have $L^h_t \le C e^{C T}$ for a constant $C$ independent of $h$. Hence, there is a constant $C'$ independent of $h$ such that
	\be \label{eq:LipschitzL_t}
		\big| u^h(t, \om) - u^h(t, \om') \big|
		&\le&
		C' \|\om_{t \wedge \cdot} - \om'_{t\wedge \cdot}\|,
		~~\forall t \in [0,T], ~\om, \om' \in \Om.
	\ee

	\vspace{1mm}
	
	We next consider the regularity of $u^h$ in $t$. Let $t := i h $ and $t' := j h>t$.
	Note that
	\beaa
		u^h(t,\om) = u^h(t+h, \om_{t \wedge \cdot}) + h G(t, \om, 0,0,0) + h \big( G(t, \om, \Dc_h u^h_{t+h} (t,\om)) - G(t, \om, 0, 0, 0) \big).
	\eeaa
	By a direct computation, we have
	\be \label{eq:uh_Xh}
		u^h(t, \om)
		&=&
        \tilde \E \left[~ \sum_{k = i}^{j-1} G(t_k, \om\otimes_t \widehat X^h, 0, 0, 0 ) ~h
		~+~
		u^h(t', \o\otimes_t\widehat X^h) ~\right],
	\ee
	where $X^h$ is a discrete process defined as $X^h_0 := 0$,
	\beaa
		X^h_{t_{k+1}}
		&:=&
		X^h_{t_k}
		~+~
		\Phi_h( \nabla_{\gamma} G , \nabla_{z} G, U_{k+1}),
	\eeaa
    with $\Phi_h$ be given by \eqref{eq:Phi_h_DF},
    and $\widehat X^h$ is the linear interpolation of $X^h$.
    Define
\beaa
A^h_0:=0,\q A^h_{t_k} := \sum_{i=0}^{k-1} \tilde \dbE\big[\Phi_h( \nabla_{\gamma} G , \nabla_{z} G, U_{i+1})\big| \tilde\cF_i\big],\q
\mbox{and}\q M^h := X^h-A^h.
\eeaa
    Clearly, $M^h$ is a martingale and $A^h$ is a predictable process.
    Further, it follows from the property of $\F_h$ in \eqref{eq:Zeta_DF} that
	\beaa
		\tilde \E \Big| A^h_{t_{k+1}} - A^h_{t_k} \Big| \le L h
		&\mbox{and}&
		\mathrm{Var} \Big[  M^h_{t_{k+1}} - M^h_{t_k} \Big] \le L h.
	\eeaa 	
Then by \eqref{eq:uh_Xh}, we have
	\be \label{eq:MartEstim}
		| u(t, \om) - u(t', \om_{t \wedge \cdot}) |
		&\le&
		C (t' - t)
		~+~ C \tilde \E \Big[\sup_{i \le k \le j} \big| M_{t_k} \big| \Big].
	\ee
	Further, by Doob's inequality, it follows that
	\beaa
		\tilde \E \Big[\sup_{i \le k \le j} \big| M^h_{t_k} \big| \Big]
		~\le~
		\sqrt{\tilde \E \Big[\sup_{i \le k \le j} \big| M^h_{t_k} \big|^2 \Big]}
		~\le~
		2 \Big(\sqrt{\tilde \E \big[ (M^h_{t_j})^2 \big]} \Big)
		~\le~
		C \sqrt{t_j - t_i}.
	\eeaa
	Finally, combining the above estimation with \eqref{eq:LipschitzL_t} and \eqref{eq:MartEstim}, we obtain \eqref{eq:uh_unif_conti}.
	\qed

	\begin{rem}{\rm
		We here assume that the PPDE is non-degenerate ($\nabla_{\gamma}G \ge \eps > 0$).
		When $\nabla_{\gamma} G= 0$ and $\nabla_{z} G\ge 0$, the scheme is still monotone.
		When $\nabla_{\gamma} G= 0$ and $\nabla_{z} G\le 0$,
		it is possible to redefine the first order discrete derivative by
		\beaa
			\Dc^1_h \psi(t,\om) ~:=~ \frac{ \psi( \om_{t \wedge \cdot}) - \psi( \om\ox_t^h (-\Delta x))}{\Delta x}		
		\eeaa
		to obtain a monotone scheme.
		}
	\end{rem}

	\begin{rem}{\rm
		In the multidimensional case, $\nabla_{\gamma} G$ is a matrix.
		If $\nabla_{\gamma} G$ is diagonal dominated, then following Kushner and Dupuis \cite[Chapter 5.3]{KushnerDupuis}, it is easy to construct a monotone scheme under similar CFL condition \eqref{eq:CFL}.
		When $\nabla_{\gamma} G$ is not diagonal dominated, it is possible to use the generalized finite difference scheme proposed by Bonnans, Ottenwaelter and Zidani \cite{BonnansOttenwaelterZidani}.
		}
	\end{rem}

\subsection{The trinomial tree scheme of Guo-Zhang-Zhuo \cite{GuoZhangZhuo}}

	We consider the PPDE of the form \eqref{eq:PPDE}.
    Let $\sigma_0$ be some symmetric $d\x d$ matrix,
	denote
	$$
		F(t, \om, y, z, \gamma) ~:=~ G(t, \om, y, z, \gamma) - \frac{1}{2} \sigma_0^2 : \gamma,
		~~ \tilde G_{\gamma} ~:=~ \sigma^{-1}_0 G_{\gamma} \sigma_0^{-1}.
	$$
	Let $\zeta = (\zeta_1, \cdots, \zeta_d)$ a random vector defined on a probability space $(\tilde\O,\tilde \cF,\tilde \P)$ such that $\zeta_i, i = 1,\cdots, d$ are i.i.d and
	$$
		\tilde\P( \zeta_i = \frac{1}{\sqrt{p}} ) = \frac{p}{2},
		~~\tilde \P(\zeta_i = -\frac{1}{\sqrt{p}}) = \frac{p}{2},
		~~\tilde \P(\zeta_i = 0) = 1-p,
        ~~\mbox{with}~ p \in (0,1).
	$$
	For every $\cF_{t+h}$-measurable function $\psi: \Om \to \R$,
    let us define $\Dc^i_h \psi(t, \om) := \tilde \E\Big[ \psi\big( \om \ox_t^h (\sqrt{h} \sigma_0 \zeta) \big) K_i(\zeta) \Big]$ with
	\beaa
		K_0 := 1, ~~ K_1 := \frac{\sigma_0^{-1} \zeta}{\sqrt{h}},
		~~K_2 := \frac{\sigma_0^{-1} [ (1-p) \zeta \zeta^T - (1 - 3p) \mbox{Diag}[\zeta \zeta^T] - 2 p I_d ] \sigma_0^{-1}}{(1-p)h},
	\eeaa
    where for any matrix $\gamma = [\gamma_{i,j}]_{1\le i,j\le d} \in S_d$, $\mbox{Diag} [\gamma]$ denotes the diagonal matrix whose $(i,i)$-th component is $\gamma_{ii}$.
	Then the numerical scheme is defined as
	\be \label{eq:TrinomialTreeScheme}
		\T^{t, \om}_h [ u^h(t+h, \cdot)]
		&:=&
		\Dc^0_h u^h(t, \om)
		~+~
		h F \big( \cdot, \Dc_h u^h_{t+h} \big) (t, \om).
	\ee

	\begin{prop} \label{prop:trinomtree}
		Let Assumptions \ref{ass:PPDE} hold true and $G$ is Lipschitz in $\om$.
		Suppose in addition that Assumption 3.3 in Guo, Zhang and Zhuo \cite{GuoZhangZhuo} holds true
		(where we replace their notation $\tilde G_{\gamma}$ by $\nabla_{\gamma} \tilde G$ in our context).
		Then the trinomial tree scheme \eqref{eq:TrinomialTreeScheme} satisfies Assumption \ref{ass:NumScheme}.
	\end{prop}
	\proof The consistency and monotonicity condition in Assumption \ref{ass:NumScheme} \rmi and \rmii
	can be justified by almost the same argument as in \cite{GuoZhangZhuo}.
	Similarly to the finite difference scheme, the monotonicity in sense of Barles and Souganidis \cite{BarlesSouganidis} implies the interpretation of the controlled discrete processes of the numerical scheme,
	which implies the monotonicity condition \eqref{eq: monotoPPDE} in our context.
	Further, using the same argument as in Proposition \ref{Prop:SchemaDF},
	it is easy to show that
	\beaa
		\big| u^h(t, \om) - u^h(t', \om') \big|
		&\le&
		C \Big (  \| \om_{t \wedge \cdot} - \om'_{t' \wedge \cdot}\| + \sqrt{|t' - t|} \Big),
		~~\forall (t,\om), (t', \om') \in \Theta,
	\eeaa
	for some constant $C$ independent of $h$,
	which implies in particular \rmiii of Assumption \ref{ass:NumScheme}.
	\qed
	
	\begin{rem}{\rm
		As a PPDE degenerates to be a classical PDE, the conditions in Proposition \ref{prop:trinomtree}
		turns to be exactly the same conditions in Theorem 3.10 of \cite{GuoZhangZhuo}.
		}
	\end{rem}

\subsection{The probabilistic scheme of Fahim-Touzi-Warin \cite{FTW}}
\label{subsec:scheme_FTW}

	We consider PPDE \eqref{eq:PPDE} in which $G$ is in the form of
	\beaa
		G(t,\om, y, z, \gamma)
		&=&
		\mu(t, \om) \cdot z ~-~ \frac{1}{2} \sigma\sigma^T(t,\om) : \gamma
		~-~ F(t, \om, y, z, \gamma).
	\eeaa
	Before introducing the numerical scheme, we first define a random vector
	\beaa
		X^{(t, \om)}_h
		&:=&
		\mu(t, \om) h ~+~ \sigma(t, \om) W_h,
	\eeaa
	where $ W_h \sim N(0, h I_d)$ is a Gaussian vector.
	For every bounded function $\psi\in \dbL^0(\cF_{t+h})$, we define
	\beaa
		\Dc_h \psi (t, \om) &:=& \E \Big[ \psi(\om \ox_t \Xh^{(t, \om)}_{\cdot}) H_h(t, \om) \Big],
	\eeaa
	where $H_h(t,\om) = (H^h_0, H^h_1, H^h_2)^T$ with
	\beaa
		H^h_0 := 1,
		&
		H^h_1 := (\sigma^T(t,\om))^{-1} \frac{W_h}{h},
		&
		H^h_2 := (\sigma^T(t,\om))^{-1} \frac{W_h W_h^T - h I_d}{h^2} \sigma^{-1}(t,\om).
	\eeaa
	Then the probabilistic scheme is given by
	\be \label{eq:SchemaFTW}
		\T^{t,\o}_h [u^h(t+h,\cd)]
		~:=~
		\E \Big[ u^h(t+h, \widehat X^{(t,\om)}) \Big] + h F\big( \cdot, \Dc_h u^h_{t+h} \big) (t, \om).
	\ee

	\begin{rem}{\rm
		The probabilistic scheme in \cite{FTW} is inspired by the second order BSDE theory of
		Cheridito, Soner, Touzi and Victoir \cite{CheriditoSonerTouziVictoir}, and extends the classical numerical scheme of BSDE
		(see e.g. Bouchard and Touzi \cite{BouchardTouziBSDE}, Zhang \cite{ZhangBSDENum}). In practice, one can use the simulation-regression method to estimate the conditional expectation in the above scheme (see e.g. Gobet, Lemor and Warin \cite{GobetLemorWarin}).
		We refer to Guyon and Henry-Labord\`ere \cite{Guyon_HenryLabordere} for more details on the use of the scheme, to Tan \cite{TanSplitting} for an extension to a degenerate case, 	and to Tan \cite{TanCtrl} for an extension to path-dependent control problems.
		}
	\end{rem}

	\begin{assum} \label{eq:AssumptionF}
		\rmi The nonlinearity $F$ is Lipschtiz w.r.t. $(\om, u, z, \gamma)$ uniformly in $t$
		and $|F(\cdot, \cdot, 0, 0, 0) |_0 < \infty$.
		
		\noindent \rmii $F$ is elliptic and dominated by the diffusion term of $X$, that is,
		\be
			\nabla_{\gamma} F \le \sigma \sigma^T,
			&\mbox{on}& \Om \x \R \x \R^d \x \S_d.
		\ee
		
		\noindent \rmiii $\nabla_p F \in Image(\nabla_{\gamma}F)$
        and $\big|(\nabla_p F)^T (\nabla_{\gamma} F)^{-1} \nabla_p F \big|_0 < \infty$.
		
		\noindent \rmiv $|\mu|_1, |\sigma|_1 < \infty$ and $\sigma$ is invertible and $\xi$ is bounded Lipschitz.
	\end{assum}

	\begin{prop}
		Suppose that Assumption \ref{eq:AssumptionF} holds true.
		Then the probabilistic numerical scheme \eqref{eq:SchemaFTW} satisfies Assumption \ref{ass:NumScheme}.
	\end{prop}	
	\proof
		\rmi Assumption \ref{ass:NumScheme} \rmi is obviously satisfied in view of Lemma 3.11 of  \cite{FTW}.
		
		\noindent \rmii Further, using probabilistic interpretation of this scheme in Tan \cite[Section 3.2]{TanCtrl}, we may verify \rmii of Assumption \ref{ass:NumScheme}.
		See also the estimation given by Lemma 3.1 of \cite{TanCtrl}.

		\noindent \rmiii	
		For \rmiii of Assumption \ref{ass:NumScheme},
		we shall prove that the numerical solution $u^h$ is Lipschitz in $\om$ and $1/2$-H\"older in $t$.
		In \cite{FTW}, the authors proved  this property in the case of PDEs. Their arguments for the Lipschitz continuity in $\om$ can be easily adapted to this path-dependent case.
		For the regularity of $u^h$ in $t$, they used a regularization technique, which seems impossible to be adapted to the path-dependent case.
		However, we can still use similar arguments as in Proposition \ref{Prop:SchemaDF}, i.e. use the discrete-time controlled semimartingale interpretation, to prove the  H\"older property of $u^h$ in $t$.
	\qed

	\begin{rem}{\rm
	As a PPDE degenerates to be a PDE, the conditions in Assumption \ref{eq:AssumptionF} reduce exactly the same conditions as in \cite{FTW} (see their Theorem 3.6).
		}
	\end{rem}

\subsection{The semi-Lagrangian scheme}
	
	For the semi-Lagrangian scheme, we shall consider the PPDE \eqref{eq:PPDE} of the Bellman-Issac type, i.e.
	the function $G$ is in the form of
	\beaa
		~G(t,\om, y, z, \gamma)
		~=~
		\inf_{k_1 \in K_1} \sup_{k_2 \in K_2}
		\Big(
			\frac{1}{2} a^{k_1, k_2}(\cdot) : \gamma
			+ b^{k_1, k_2}(\cdot) \cdot z
			+ c^{k_1, k_2}(\cdot) y
			+ f^{k_1, k_2}(\cdot)
		\Big)(t,\om),
	\eeaa
	where $K_1$ and $K_2$ are some sets, $(a^{k_1, k_2}, b^{k_1, k_2}, c^{k_1,k_2}, f^{k_1, k_2})$ are functionals defined on $\Theta$.

	Let $\zeta$ be a random vector satisfying
	\be \label{eq:semiLag_xi}
		\E \big[ \zeta \big] = 0,
		~~~
		\mbox{Var} \big[ \zeta \big] = I_d
		&\mbox{and}&
		\E \big[ \big| \zeta \big|^3 \big] < \infty.
	\ee
	Then the semi-Lagrangian scheme is defined as
	\be \label{scheme:semiLagrangian}
		\T^{t,\o}_h[ u^h(t+h,\cd)]
		&:=&
		\inf_{k_1 \in K_1} \sup_{k_2 \in K_2}
		\Big \{
			u^h \big(t+h, \om \otimes_t \big(\si^{k_1,k_2}(t,\om) \zeta \sqrt{h} + b^{k_1, k_2}(t,\om) h  \big)\big) \nonumber \\
		&&~~~~~~~~~~~~~+
			u^h \big(t+h, \om) c^{k_1, k_2}(t,\om) h
			~+~
			f^{k_1, k_2}(t,\om) h
		\Big \}.
	\ee

	\begin{prop}
		Suppose that $|a|_1 + |b|_1 + |c|_1 + |f|_1 < \infty$,
		and \eqref{eq:semiLag_xi} holds true.
		Then the semi-Lagrangian  scheme \eqref{scheme:semiLagrangian} for the Bellman-Issac path-dependent equation satisfies
		Assumption \ref{ass:NumScheme}.
	\end{prop}
	\proof \rmi The consistency condition (Assumption \ref{ass:NumScheme} (i)) is easy to check.
	
	\noindent \rmii Let $E$ be a set, $e: K_1 \x K_2 \to E$ be an arbitrary mapping, and
	$\psi, \varphi : E \to \R$ be two bounded functions. Note that
	\be\label{eq:Ineq_SL}
		\inf_{k_1 \in K_1} \sup_{k_2 \in K_2} \psi( e(k_1,k_2) )-
		\inf_{k_1 \in K_1} \sup_{k_2 \in K_2} \varphi( e(k_1,k_2) )
		\le
		\sup_{k_1 \in K_1, k_2 \in K_2} (\psi-\varphi)( e(k_1,k_2) ).
	\ee
    Notice that $\R^d$ is isomorphic to $\R$, we can always consider the random vector $\sigma^{k_1,k_2}\zeta \sqrt{h} + b^{k_1, k_2} h $ as a one-dimensional random variable.
    By consider the inverse function of its distribution function, then there is a family
    $\Phi_h(k_1, k_2, \cdot)$ such that $\Phi^h(k_1, k_2, U) \sim \sigma^{k_1,k_2}\zeta \sqrt{h} + b^{k_1, k_2} h $ in law with $U \sim \Uc([0,1])$, for all $(k_1, k_2) \in K_1 \x K_2$.
	Then it follows from \eqref{eq:Ineq_SL} that the monotonicity condition in Assumption \ref{ass:NumScheme} \rmii holds true with $\Phi_h(k_1, k_2, \cdot)$ and $K = K_1 \x K_2$.
	
	\noindent \rmiii Finally, by the same arguments as in Proposition \ref{Prop:SchemaDF}, we can easily deduce that
	$u^h$ is Lipschitz in $\om$ and $1/2$-H\"older in $t$, uniformly on $h$,
	and hence complete the proof for the stability condition in Assumption \ref{ass:NumScheme}.
	\qed

	\begin{rem}{\rm
		Solutions of path dependent Bellman-Issac equations can characterize value functions of stochastic differential games (see e.g. Pham and Zhang \cite{PhamZhang}).
}	
	\end{rem}

	\begin{rem}{\rm
		\rmi For Bellman-Issac PDE, Debrabant and Jakobsen \cite{Debrabant_Jakobsen} studied the semi-Lagrangian scheme with a random variable $\zeta$ following a discrete distribution, together with an interpolation technique for the implementation.

\ms
		\noindent \rmii
		For Bellman equation (PDE), Kharroubi, Langren\'e and Pham \cite{KharroubiLangrenePham}
		propose a semi-Lagrangian type numerical scheme with $\zeta \sim N(0,1)$,
		and provide a simulation-regression technique for the implementation.
		It is worth of mentioning that \cite{KharroubiLangrenePham} provides a convergence rate for the scheme,
		while we only prove in this paper a general convergence theorem as in Barles and Souganidis \cite{BarlesSouganidis}.
		}
	\end{rem}

\section{Numerical examples}

	We will provide two toy examples of numerical implementation in low-dimensional case.
	Notice also that the main focus of the paper is on a general convergence theorem for numerical analysis of non-Markovian control problems, 
	and we will not propose any new numerical schemes.
	For more numerical examples (in high-dimensional case) of difference numerical schemes, we would like to refer to \cite{FTW, GuoZhangZhuo, Guyon_HenryLabordere, KharroubiLangrenePham, TanSplitting}, etc.

	In our two toy examples, we implement the finite difference scheme in Section \ref{subsec:schemeFD} and the probabilistic scheme in Section \ref{subsec:scheme_FTW}.
	In this low dimensional case, the finite difference scheme is quite easy to be implemented given some boundary conditions.	For the probabilistic scheme, we do not need the boundary conditions,
	but need a simulation-regression technique (as studied in Gobet, Lemor and Warin \cite{GobetLemorWarin}) to estimate the conditional expectations appearing in the scheme.
	More concretely, we will use the local polynomial functions as regression basis, as introduced by Bouchard and Warin \cite{BouchardWarin}.

\paragraph{A first numerical example}

	For a first numerical example, we consider the PPDE
	\be \label{eq:PPDE_Exam1}
		-\pa_t u - \min_{\mu \in [\ul \mu,\ol \mu]} \mu \pa_\o u -\max_{a\in [\ul a,\ol a]}\frac{a}{2}\pa^2_{\o\o}u = f(t,\o,\bar\o), \q u(T,\o) = g(\o_T,\bar \o_T).
	\ee
	where $d = 1$, $\bar\o_t := \int_0^t \o_s ds$, $f:[0,T] \x \R \x \R \to \R$ and $g: \R \x \R \to \R$ are two functions.
 In this case, the value function dependents only on $(t, \om_t, \omb_t)$, so that the numerical solution can be written as $u^h(t, \om_t, \omb_t)$.

	The above PPDE \eqref{eq:PPDE_Exam1} is motivated by a stochastic differential game:
	\beaa
		u_0 = \inf_{\ul\mu\le \mu_t\le \ol\mu}\sup_{\ul a\le a_t\le \ol a}\dbE\Big[\int_0^T f(t,X^{\mu,a}_t, \ol X^{\mu, a})dt +g(X^{\mu,a}_T,\ol X^{\mu,a}_T)\Big],
	\eeaa
	where $X^{\mu,\si}$ is controlled diffusion such that
	\beaa
		X_t^{\mu,a} = \int_0^t \mu_s ds + \int_0^t \sqrt{a_s}dW_s,
		\q
		\mbox{with $W$ a Brownian motion},
	\eeaa
	and $\ol X_t^{\mu,a} = \int_0^t X_s^{\mu,a} ds$ (see e.g. Pham and Zhang \cite{PhamZhang} for more details).

	We choose the terminal condition $g(x,y)=\cos (x+y)$ and the function
	\beaa
		f(t,x,y) 
		&=&
		-~ (x -\ul \mu) \big(\sin (x-y)\big)^- 
		~+~ 
		(x+\ol\mu) \big(\sin (x-y)\big)^+ \\
		&&
		+~ 
		\frac{\ul a}{2}\big(\cos(x-y)\big)^+ 
		~-~
		 \frac{\ol a}{2}\big(\cos(x-y)\big)^-,
	\eeaa
	so that the solution to PPDE \eqref{eq:PPDE_Exam1} is given explicitly by $u(t,\o)= \cos (\o_t+\bar\o_t)$,
	which serves as a reference value for the numerical examples (this idea is borrowed from Guo, Zhang and Zhuo \cite{GuoZhangZhuo}).
	 For finite difference scheme, we use $\cos(\o_t + \bar \o_t)$ as boundary conditions since it is the exact solution of the PPDE.
	For the probabilistic scheme, we simulate $100~000$ paths $(X^k_{\cdot})_{k = 1, \cdots, 100~000}$ of a diffusion process on discrete time grids.
	The diffusion process $X$ is defined by $dX_t = \sqrt{\al} dW_t$, with a Brownian motion  $W$, and also denote $A_t := \int_0^t X_s ds$.
	Then for regression function basis, we use the local polynomial of order 2  (as introduced in \cite{BouchardWarin}), i.e.
	$1, \om_t, \omb_t, \om_t^2, \omb_t^2, \om_t \omb_t$ on each local hypercube.
	To define the local hypercube, we will find the minimum and maximum of $X_t$ as well as $A_t$ of all simulations, and then divide uniformly the domain $[\min_k X^k_t, \max_k X^k_t] \x [\min_k A^k_t, \max_k A^k_t]$ into $20 \x 28$ small hypercubes.

	\begin{figure}[htb!]
		\centering
		\includegraphics[scale=0.9]{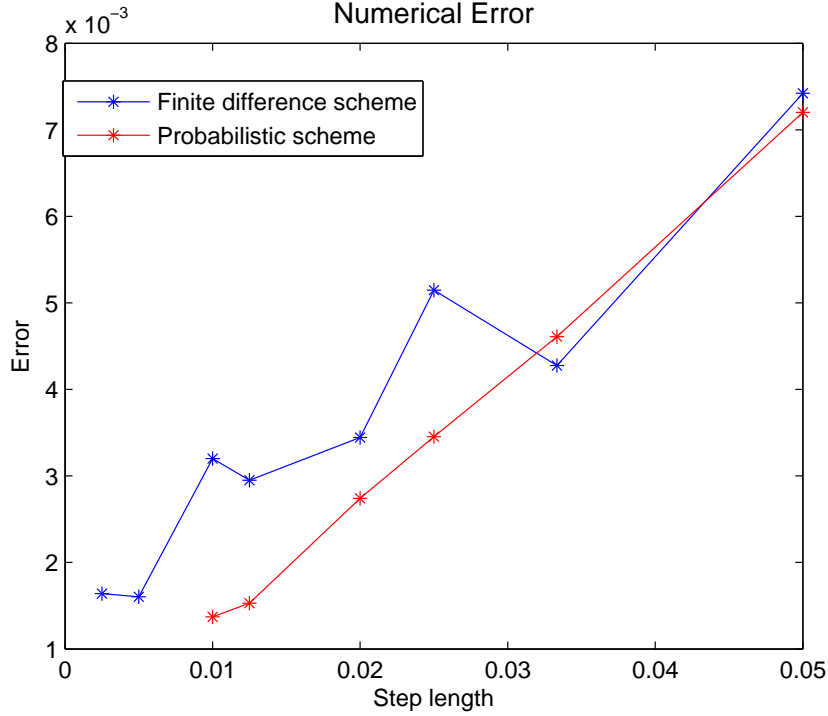}
		\caption{For PPDE \eqref{eq:PPDE_Exam1}, we choose $\ul \mu=-0.2$, $\ol\mu =0.2$, $\ul a=0.04$, $\ol a=0.09$, $T=1$ and $\om_0 = \omb_0 = 0$. Then the reference solution is given by $u(0, 0) = \cos(0) = 1$.
		We compute the error between the reference solution and the numerical solutions, w.r.t. difference time step length $\Delta t$.
		}
		\label{fig:PPDE1}
	\end{figure}

\paragraph{A second numerical example}
	
	The second example of PPDE we considered is given by
	\be \label{eq:PPDE_Exam2}
		&-\pa_t u -\max_{\ul a\le a\le \ol a}\Big(\frac12 a\pa^2_{\o\o} u-f(t,u,\pa_\o u,a)\Big) ~=~ 0,&\\
		&\mbox{where}~~f(t,y,z,a)=\frac12\big((\sqrt{a}z+ b/ \sqrt{a})^- \big)^2-z b-b^2/2a,&  \nonumber
	\ee
	which is taken from Matoussi, Possamaï and Zhou \cite{MPZ}.
	The above equation is motivated by solving a robust utility maximization problem using 2BSDE,
	which can be instead characterized by a PPDE (see e.g. \eqref{eq:PPDE_2BSDE}).

	We consider the terminal condition 
	\beaa
		u(T,\o)=K_1+(\bar\o_T - K_1)^+ -(\bar\o_T -K_2)^+,
		~~~\omb_T := \int_0^T \om_s ds.
	\eeaa
	Then the solution to PPDE \eqref{eq:PPDE_Exam2} can also be characterized by the PDE, by adding an associated variable $y$,
	\be \label{eq:PDE_exam2}
		&-\pa_t v - x \pa_y v -\max_{\ul a\le a\le \ol a}\Big(\frac12 a\pa^2_{xx} v-f(t,v,\pa_x v,a)\Big) =0,&\\
		& v(T,x,y) = K_1+(y - K_1)^+ -(y-K_2)^+.& \nonumber
	\ee

	We implemented the finite difference scheme (Section \ref{subsec:schemeFD}) and the probabilistic scheme (Section \ref{subsec:scheme_FTW}) for PPDE \eqref{eq:PPDE_Exam2}.
	As reference, we implemented the classical finite difference scheme of PDE \eqref{eq:PDE_exam2}.
	Here for the finite difference scheme, we fixe the computation domain as $[-0.8, 0.8] \x [-0.8, 0.8]$, and use a fictive Neumann boundary conditions $\partial_x u(t,x,a) = 0$ for $x = \pm 0.8$ and a fictive Dirichlet boundary conditions $u(t,x, -0.8) = K_1$ and $u(t,x, 0.8) = K_2$.
	For the probabilistic scheme, we use the same simulation-regression techniques as previous PPDE to estimate the conditional expectations appearing in the scheme.
	We also notice that the generator in PPDE \eqref{eq:PPDE_Exam2} is in fact not Lipschitz but quadratic in $z$, 
	however, the convergence of the numerical solutions can be still observed, see Figure \ref{fig:PPDE2}.

	\begin{figure}[htb!]
		\centering
	\includegraphics[scale=0.9]{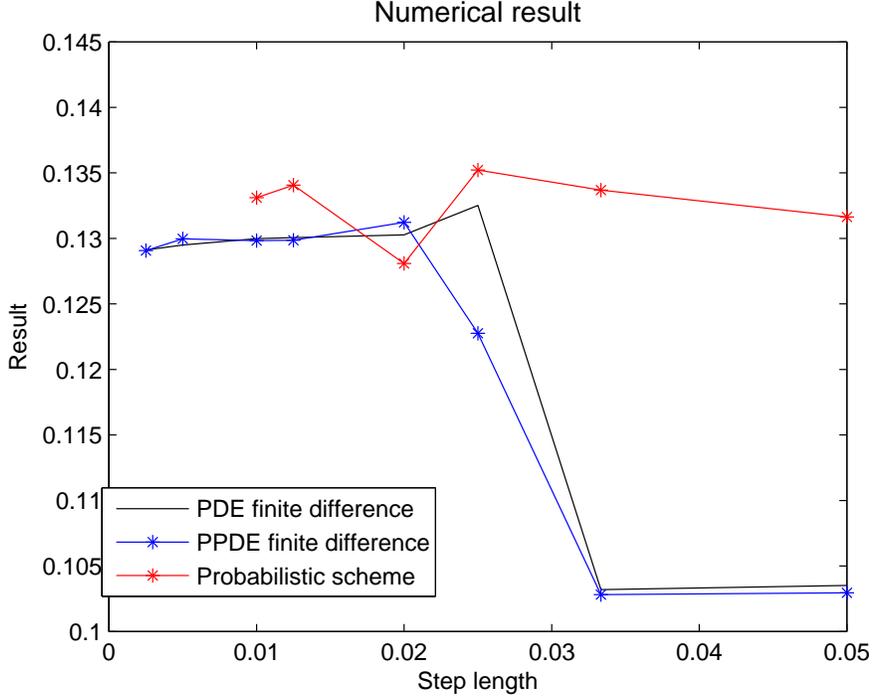}
		\caption{For PPDE \eqref{eq:PPDE_Exam1}, we choose $K_1=-0.2$, $K_2 = 0.2$, $\ul a=0.04$, $\ol a = 0.09$, $b=0.05$ and $T=1$.
		We provide all the numerical solutions w.r.t. difference time step length $\Delta t$.
		It seems that the faire value is closed to 0.129.
		For finite-difference scheme, when $\Delta t$ is greater than 0.025, we need to use a coarser space-discretization to ensure the monotonicity (similar to the classical CFL condition), which makes a big difference to the numerical solutions for the case $\Delta t < 0.25$.
		However, the convergence as $\Delta t \to 0$ is still obvious.
		}
		\label{fig:PPDE2}
	\end{figure}

\section{Proofs}\label{sec:proofs}

\subsection{Preliminary results}

In preparation of the proof of Theorem \ref{thm:main}, we prove the following lemmas.

\begin{lem}[Fatou's Lemma]\label{lem: Fatou}
Assume that the random variables $X^n\in C^0(\cF)$ are uniformly bounded.  Then we have
\beaa
\liminf_{n\rightarrow\infty} \ul\cE [X^n] \ge \ul\cE \big[\liminf_{n\rightarrow\infty} X^n\big]
\eeaa
\end{lem}
\begin{proof}
In order to prove Fatou's lemma, it is enough to show the monotone convergence theorem, i.e. given a sequence $\{X^n:n\in\dbN\}$ of increasing random variables,  we have
    \be \label{eq:MonoCvg}
        \lim_{n\rightarrow\infty} \ul\cE[X^n]=\ul\cE[\lim_{n\rightarrow\infty} X^n].
    \ee
    Since $X^n\in C^0(\cF)$ for each $n$, it follows from Theorem 31 in \cite{DenisHuPeng} that \eqref{eq:MonoCvg} holds true.
\end{proof}

Recall the nonlinear expectation $\ul\cE_h$ defined in \eqref{defn: cEh}.
	\begin{lem}\label{lem: eqv convergence}
		Let $\f:\O\rightarrow\dbR$ be bounded uniformly continuous.
		Then there exists a modulus of continuity $\rho : \R_+ \to \R_+$ which depends only on the moduli of continuity of $\f$ and $|\f|_0$,
		such that
		\beaa
			\ul\cE[\f] &\le& \ul\cE_h[\f] + \rho(h).
		\eeaa
	\end{lem}
	\begin{proof}
		Denote $\rho' : \R_+ \to \R_+$ as a continuity modulus of $\f$.
		Let $\nu \in \Kc$ and $X^{h,\nu}$ be defined by \eqref{eq:def_Xh} and $\widehat X^{h,\nu}$ its linear interpolation on $[0,T]$.
		Then under the condition \eqref{eq:CondH}, it follows from Lemma 4.8 of Tan \cite{TanCtrl} (see also Dolinsky \cite{Dolinsky}) that
		we can construct a process $\widehat X^{h,\nu}$ and another process $\overline X$ in the same probability space $(\tilde \Om, \tilde \Fc, \tilde \P)$, such that the image measure of $\overline X$ lies in $\Pc$,
		and for some constant $C$ independent of $h$,
		\beaa
			\tilde \P \Big( \big| \widehat X^{h,\nu} - \overline X \big| \ge h^{1/8} \Big)
			&\le&
			C h^{1/8}.
		\eeaa
		Let $\rho (h) := \rho'(h^{1/8}) + 2 \| \varphi \|_\infty h^{1/8}$, then it follows that
		\beaa
			\ul\cE[\f]
			\le
			\tilde \dbE[\f(\ol X)]
			\le
			\tilde\E \big[ \f(\widehat X^{h, \nu}) \big] + \rho(h),
		\eeaa
		which concludes the proof by the arbitrariness of $\nu \in \Kc$.
		\qed
	\end{proof}

\begin{lem}\label{lem: usc conv}
Let $\f:\O\rightarrow \dbR$  be lower semicontinuous and bounded from below, then it holds for all $(t,\o)\in \Th$ that
\beaa
\liminf_{h\rightarrow 0} \ul\cE_h[\f]
~\ge ~
\ul\cE[\f].
\eeaa
In particular, by defining $\ch : =\inf\{t\ge 0: |B_t|\ge x\}$ for some $x>0$, we have
\beaa
\limsup_{h\rightarrow 0} \ol\cE_h[1_{\{\ch \le \d\}}]
~\le ~
\ol\cE[1_{\{\ch\le \d\}}] \q\mbox{for any}~~\d>0.
\eeaa
\end{lem}

\begin{proof}
Define the approximation for the function $\f$:
\beaa
\f^n(\o):=\inf_{\o'\in\O}\big\{\f(\o')+n\|\o-\o'\|\big\}.
\eeaa
Clearly, for each $n\in \dbN$, function $\f^n$ is Lipschitz continuous, and $\f^n \uparrow \f$. By Lemma \ref{lem: eqv convergence}, we obtain that
\beaa
\liminf_{h\rightarrow 0}\ul\cE_h[\f]
    ~ \ge ~ \limsup_{h\rightarrow 0}\ul\cE_h[\f^n]
    ~ \ge ~   \ul\cE [\f^n], \q\mbox{for all}~~n\in\dbN.
\eeaa
Since $\f^n\uparrow\f$, by Fatou's lemma we have
\beaa
\liminf_{n\rightarrow\infty} \ul\cE[\f^n]
 ~\ge ~ \ul\cE[\f].
\eeaa
Therefore
\be\label{eq: res1}
\liminf_{h\rightarrow 0}\ul\cE_h[\f]  ~\ge~  \ul\cE[\f].
\ee

Then we easily get the symmetric result for upper semicontinuous function $\psi$, i.e.
\beaa
\limsup_{h\rightarrow 0} \ol\cE_h [\psi]~\le~ \ol\cE[\f].
\eeaa
To conclude, it remains to prove that the function $\o\longmapsto 1_{\{\ch(\o)\le \d\}}$ is upper semicontinuous. Note that
\beaa
\{\ch\le \d\} = \{\max_{t\in [0,\d]} |B_t|\ge x\}
\eeaa
Since the function $\f: \o\mapsto \max_{t\in [0,\d]} |B_t(\o)|$ is continuous, the set $\{\ch\le \d\}$ is closed. Consequently, the function $1_{\{\ch\le \d\}}$ is upper semicontinuous.
\qed
\end{proof}

\begin{lem}\label{lem: Hdelta}
For any $\d>0$ and $\e>0$, define $x(\d)= Ld\sqrt{\d}\big(\sqrt{\d} + \sqrt{-2\ln\frac{\e\d}{4d}}\big)$ and $\ch^{\d} = \inf\{t\ge 0: |B_t| \ge x(\d) \}$. Then,  
for $\d$ small enough we have
\be\label{defn: Hdelta}
\sup_{\dbP\in \cP}\dbP[\ch^\d\le \d] \le \e\d.
\ee
\end{lem}

\begin{proof}
Note that
\beaa
\sup_{\dbP\in \cP}\dbP[\ch^\d\le \d]  = \sup_{\dbP\in \cP}\dbP\big[\max_{t\in [0,\d]}|B_t|\ge x\big]
\le d \sup_{\dbP\in \cP}\dbP\big[\max_{t\in [0,\d]}|B^1_t|\ge \frac{x}{d}\big]
\eeaa
    By the definition of $\cP$ above \eqref{defn:PL}, for all $\P \in \cP$, the canonical process $B$ admits the canonical decomposition $B = A^{\P} + M^{\P}$,
    where $A^{\P} =(A^1,\cds,A^d)$ is a finite variation process
    and $M=(M^1,\cds,M^d)$ is a $\P$-martingale.
    Moreover, for each $i = 1, \cdots, d$,
\beaa
\dbP\big[\max_{t\in [0,\d]}|B^i_t|\ge \frac{x}{d}\big]
 =  \dbQ\big[\max_{t\in [0,\d]}|A^i_t+M^i_t|\ge \frac{x}{d}\big]
\le \dbQ\big[\max_{t\in [0,\d]}|M^i_t|\ge \frac{x}{d} - L\d\big].
\eeaa
Further, by the time-change for martingales (see e.g. Theorem 4.6 on page 174 of \cite{KS}), there is a scalar Brownian motion $W$ defined on a probability space $(\fO,\fF,\fP)$ such that
\beaa
\P\big[\max_{t\in [0,\d]}|M^i_t|\ge \frac{x}{d} - L\d\big]
 & = & \fP\big[\max_{t\in [0,\d]}|W_{<M^1>_t}|\ge \frac{x}{d} - L\d\big]\\
 &\le &  \fP\big[\max_{t\in [0,L^2\d]}|W_t|\ge \frac{x}{d} - L\d\big]\\
 & = & 4\fP\big[W_1 \ge \frac{x/d-L\d}{L\sqrt \d}\big]
\eeaa
Since $\eta: = \frac{x/d-L\d}{L\sqrt \d} = \sqrt{-2\ln\frac{\e\d}{4d}}>1$ when $\d$ is small enough,  we have
\beaa
4\fP\big[W_1 \ge \eta\big]
\le   4 e^{-\frac{\eta^2}{2}}  =  \frac{\e\d}{d}.
\eeaa
We then conclude that $\sup_{\dbP\in \cP}\dbP[\ch^\d\le \d] \le \e\d$.
\qed
\end{proof}

\subsection{ Proof of Proposition \ref{prop:eqv_defn}}

We only discuss the case of subsolution. The result about the supersolution follows similarly.

\ms
\no {\bf 1.}\q We first prove the only if part. Let $(t,\o)\in [0,T)\times \O$ and $(\a,\b,\g)\in \ul\cJ u(t,\o)$ with a localizing time $\ch_\d$. Clearly, there is a function $\f\in C^{1,2}_0(\dbR^+\times\dbR^d)$ such that $\f = \phi^{\a,\b,\g}$ on the set $[0,\d]\times\{x\in\dbR^d: |x|\le x(\d)\}$, where $x(\cd)$ is defined as in Lemma \ref{lem: Hdelta}. Thus,
\beaa
(\f - u)_0=\max_{\t\in \cT_{\bar\ch^{\d,x}}}\ol\cE [(\f-u)_\t],
\eeaa
where $\bar\ch^{\d,x} := \d\we\ch^{\d,x}$ with $\ch^{\d,x}$ be defined as in Lemma \ref{lem: Hdelta}. We have
\be\label{eq: from H to d}
(\f -u)_0 &\ge & \ol\cE [(\f-u)_\d ] - \ol\cE [(\f-u)_\d-(\f-u)_{\bar\ch^{\d,x}}].
\ee
For the second term on the right hand side of \eqref{eq: from H to d}, we have
\beaa
\ol\cE [(\f-u)_\d-(\f-u)_{\bar\ch^{\d,x}}]
		&\le &  \ol\cE \big[|(\f-u)_\d-(\f-u)_{\bar\ch^{\d,x}}|; \ch^{\d,x}\le \d \big]\\
		&\le &  C\sup_{\dbP\in\cP}\dbP[\ch^{\d,x}\le \d].
\eeaa
Take $\e>0$. By Lemma \ref{lem: Hdelta}, there is a constant $C(\e)>0$ such that for all $\d<C(\e)$ we have $\sup_{\dbP\in\cP}\dbP[\ch^{\d,x}\le \d] < \frac{\e\d}{2C}$. Then it follows from \eqref{eq: from H to d} that
\beaa
(\f-u)_0 > \ol\cE [(\f-u)_\d] -\frac{\e\d}{2}.
\eeaa
We next consider the optimal stopping problem:
\beaa
Y_t(\o) = \sup_{\t\in \cT_{\d-t}}\ol\cE [(\f-u)^{t,\o}_\t - \e\t].
\eeaa
According to Ekren, Touzi and Zhang \cite{ETZ-os}, $\t^*:= \inf\{t: Y_t = \f_t-u_t-\e t\}$ is an optimal stopping rule. Suppose that we always have $\bar\ch^{\d,x}\le \t^*\le \d$. Then we obtain that
\beaa
\ol\cE [(\f-u)_{\t^*} - \e\t^*] &\le & \ol\cE [(\f-u)_\d - \e\d] +\ol\cE [(\f-u)_{\t^*} - (\f-u)_\d -\e(\t^*-\d)] \\
		&\le & \ol\cE [(\f-u)_\d - \e\d] +\ol\cE [|(\f-u)_{\t^*} - (\f-u)_\d -\e(\t^*-\d)|; \ch^{\d,x}\le \d] \\
		&\le & \ol\cE [(\f-u)_\d - \e\d] + C\sup_{\dbP\in \cP}\dbP [\ch^{\d,x}\le \d]\\
		&\le & \ol\cE [(\f-u)_\d]- \frac{\e\d}{2} ~<~(\f-u)_0.
\eeaa
However, this is in contradiction with the optimality of $\t^*$. Therefore, there is $\o^*$ such that $t^*:=\t^*(\o^*) < \bar\ch^{\d,x}(\o^*)$ and
\beaa
(\f-u)_{t^*}(\o^*) &=& \max_{\t\in \cT_{\d -t^*}}\ol\cE [(\f-u)^{t^*,\o^*}_\t-\e\t].
\eeaa
So we have
\beaa
\big(-\pa_t \f+\e - G(\cd, u,\pa_x\f,\pa^2_{xx}\f)\big)(t^*,\o^*) &\le & 0.
\eeaa
By letting $\d\rightarrow 0$ and then $\e\rightarrow 0$, we obtain
    $$
        \big(-\pa_t \f - G(\cd, u,\pa_x\f,\pa^2_{xx}\f)\big)(0,0) ~~\le~~  0.
    $$
    Finally, since $\a = \pa_t \f_0, \b=\pa_x \f_0, \g=\pa^2_{xx}\f_0$, this provides that $-\a -G(0, u_0, \b,\g)\le 0$.

\ms
\no {\bf 2.}\q For the if part, one may apply the same argument as in Proposition 3.11 in \cite{RTZ-survey}. For completeness, we provide the full argument.  Let $(t,\o)\in[0,T)\times\Omega$ and $\varphi\in\underline\cA u(t,\omega)$ with a localizing time $\d\in\dbR^+$. Without loss of generality, we assume that $(t,\omega)=(0,0)$ and $(\varphi-u)_0=0$. Denote
 \bea
 \label{ab}
 \a:=\pa_t\f_0,\q
 \b:=\pa_x \f_0,\q
 \mbox{and}~~\g:=\pa^2_{xx} \f_0.
 \eea
For any $\e>0$, since $\f$ is smooth, by otherwise choosing a stopping time  $\ch_{\d'}<\d$ we may assume
 \beaa
 |\pa_t\f_t-\a|\le\eps,\q
|\pa_x\varphi_t-\b|\le\eps,\q
|\pa^2_{xx}\varphi_t-\g|\le2\eps,\q
 0\le t\le \ch_{\d'}.
 \eeaa
Denote $\a_\eps:=\a+[1+2L]\e$. Then, for all  $\tau\in\cT_{\ch_{\d'}}$,
 \beaa
 && \ol\cE \big[(u-\phi^{\a_\e,\b,\g})_{\tau}\big]-u_0 = \ol\cE \big[(u - u_0 - \phi^{\a_\e,\b,\g})_{\t}\big]\\
 &\le &	\ol\cE \big[(u -\f)_{\t}\big]+
 \overline\cE \!\big[(\varphi\!-\!\varphi_0\!-\!\phi^{\a_\e,\b,\g})_{\tau}\big]\\
 &\le &
 \ol\cE\Big[\int_0^{\t}(\pa_t\f_s -\! \a_\eps)ds
                                  +(\partial_x\f_s - \b) \cd dB_s
                                  +\frac12 (\pa^2_{xx}\f_s-\g):d\langle B\rangle_s
                          \Big].
 \eeaa
 where the last inequality is due to the It\^o's formula. Note that, for any $\|\mu\|_\infty,\|a\|_\infty\le L$, we have
 \begin{multline*}
  \dbE^{\dbQ_{\mu,\si}}\Big[\int_0^{\t}(\pa_t\f_s -\! \a_\eps)ds
                                  +(\partial_x\f_s - \b) \cd dB_s
                                  +\frac12(\pa^2_{xx}\f_s-\g):d\langle B\rangle_s
                          \Big]\\
 = \dbE^{\dbQ_{\mu,\si}}\Big[\int_0^{\t}
                                  \Big(\pa_t \f_s - \a
                                  + (\pa_x \varphi_s - \b) \cd  \mu_s+\frac12(\pa^2_{xx}\f_s-\g):a_s\Big)ds- [1+2L]\e\t
                          \Big]
                          \le  0.
 \end{multline*}
 By the arbitrariness of $\mu,\si$, we see that
 \beaa
\ol\cE\!\big[(u-\phi^{\a_\e,\b,\g})_{\tau}\big]-u_0 & \le & 0.
                          \eeaa
That is,  $(\a_\eps,\b)\in\underline\cJ u_0$. Since $u$ is a $\cP$-viscosity subsolution, it follows that
\beaa
-\a_\e-G(0,0,u_0,\b,\g) & \le & 0.
\eeaa
Let $\e\rightarrow 0$, then the desired result follows.
\qed

\subsection{Proof of Theorem \ref{thm:main}}

We first introduce two functions:
\be\label{defn: u bar}
\ul u (t,\o) = \liminf_{h\rightarrow 0} u^h(t,\o)
\q\mbox{and}\q
\ol u (t,\o) = \limsup_{h\rightarrow 0} u^h(t,\o).
\ee
Note that $\ul u,\ol u$ inherit the uniform modulus of continuity of $u^h$, so $\ul u,\ol u\in \mbox{\rm BUC}(\Th)$. It is also clear that $\ul u\le \ol u$ and $\ul u_T = \ol u_T$. Then it is enough to prove that $\ul u$ is a $\cP$-viscosity supersolution and $\ol u$ is a $\cP$-viscosity subsolution, so that by the comparison principle we may obtain $\ol u\le \ul u$,
to conclude the proof of Theorem \ref{thm:main}.

\begin{prop}\label{prop:supersolution}
The functions $\ul u$ and $\ol u$ defined in \eqref{defn: u bar} are $\cP$-viscosity supersolution and subsolution, respectively.
\end{prop}
\begin{proof}
We only prove the result for $\ul u$. The corresponding result for $\ol u$ can be proved similarly.

\ms
\no {\bf 1.}\q Without loss of generality, we only verify the viscosity supersolution property at the point $(0,0)$. Let function $\f\in \ol\cA \ul u(0,0)$, and by adding a constant to $\varphi$, we assume that $\ul u(0,0) > \varphi(0,0)$, so that
\be\label{eq: AL}
0 < \eta =: (\ul u-\f)_0 = \min_{\t\in \cT_\d}\ul\cE[(\ul u-\f)_\t],\q\mbox{for some}~\d>0.
\ee
Assume that $\ul u$ and $\f$ are both bounded by a constant $M\ge 0$. Take a subsequence still named as $u^h$ such that $\ul u_0 = \lim_{h\rightarrow 0} u^h_0$. Now fix a constant $\e>0$, and denote $\f^\e(t,x) = \f(t,x)-\e t$.  By Lemma \ref{lem: Hdelta}, there is a constant $C(\e)\in(0,\frac{1}{L} \we 1)$ such that for all $0<\d< C(\e)$, we have
\be\label{how small H}
\sup_{\dbP\in\cP}\dbP[\ch^{\d,x}\le \d]  ~\le ~ \frac{\e}{32(2M+\e)} \d.
\ee
Since $u^h$ is uniformly continuous uniformly in $h$, by considering $\d$ small enough we may assume that $u^h - \f^\e >0$ on $[0,\bar\ch^{\d,x}]$, where $\bar\ch^{\d,x}:=\d\we \ch^{\d,x}$. It follows from \eqref{eq: AL} that
\be\label{diff at delta}
(\ul u-\f^\e)_0 ~\le~ \ul\cE[(\ul u-\f)_\d] ~ =~ \ul\cE[(\ul u-\f^\e)_\d] - \e\d.
\ee

\no It follows from (iii) of Assumption \ref{ass:NumScheme} that $\{ u^h_\d-\f^\e_\d :h>0\}$ is uniformly bounded and uniformly continuous uniformly in $h$. By Lemma \ref{lem: Fatou} and \ref{lem: eqv convergence}, we obtain that
\be\label{ehe}
\liminf_{h\rightarrow 0}\ul\cE_h \big[  u^h_\d-\f^\e_\d \big]
&\ge &	\liminf_{h\rightarrow 0}\ul\cE\big[  u^h_\d-\f^\e_\d \big] + \liminf_{h\rightarrow 0}\big( \ul\cE_h \big[  u^h_\d-\f^\e_\d \big]-\ul\cE \big[  u^h_\d-\f^\e_\d \big]\big) \notag\\
&\ge &	\ul\cE\big[ \ul u_\d-\f^\e_\d\big] + \liminf_{h\rightarrow 0}\inf_{\ell>0} \big( \ul\cE_{h}\big[  u^\ell_\d-\f^\e_\d \big]-\ul\cE \big[  u^\ell_\d-\f^\e_\d \big] \big) \\
& \ge & \ul \cE \big[ \ul u_\d-\f^\e_\d \big] +\liminf_{h\rightarrow 0} \rho(h) 
~ = ~		\ul\cE \big[ \ul u_\d-\f^\e_\d\big]. \notag
\ee
It follows from \eqref{how small H} and \eqref{ehe} that for $h$ sufficiently small we have
\be\label{obj: contradiction}
(u^h-\f^\e)_0 & \le & \ul\cE_h\big[(u^h - \f^\e)_\d\big] - \frac{3\e\d}{4}.
\ee
Then by the optimal stopping argument in Step {\bf 2}, we may find $(t^*,\o^*)\in \Th$ such that
\be\label{obj: step4}
&\bar\ch^{\d,x}(\o^*)\we (\d-h) > t^* \in \Delta_h :=\{kh:k\in\dbN\} & \notag \\
  &\mbox{and}\q  (u^h - \f^\e)^{t^*,\o^*} _0 ~= ~ \min_{\t\in \cT^h_{\d-t^*}, \beta \in \Bc^h} \ul\cE_h[\beta_{\tau}(u^h - \f^\e)^{t^*,\o^*}_\t], &
\ee
where $\cT^h_{\d-t^*} : =\{\t\in\cT_{\d-t^*}: \t~\mbox{takes values in}~\Delta_h\}$ and
$\Bc^h$ is the collection of all processes $\beta$ defined by $\beta_t := e^{\sum_{i=0}^{[t/h]-1} \alpha_i h}$ for some $\cF_{ih}$-measurable functions $\alpha_i$ taking value in $[0,L]$.
In particular, \eqref{obj: step4} implies that
\be\label{uhfe}
(u^h - \f^\e)(t^*,\o^*) ~\le~ \inf_{0\le \a\le L}\ul\cE_h[e^{\a h}(u^h - \f^\e)^{t^*,\o^*}_h].
\ee
Since $t^*< \bar\ch^{\d,x}(\o^*)$, we have $(u^h - \f^\e)(t^*,\o^*)>0$, and thus 
$$\sup_{0\le \a\le L} e^{\a h}(u^h - \f^\e)(t^*,\o^*) ~=~  e^{Lh} (u^h - \f^\e)(t^*,\o^*).$$
 So it follows from \eqref{uhfe} that
\beaa
0 ~ \le ~  \inf_{0\le \a\le L}\ul\cE_h \Big[e^{\a h}\Big((u^h - \f^\e)^{t^*,\o^*}_h - e^{-Lh} (u^h - \f^\e)(t^*,\o^*)\Big)\Big]
\eeaa
By (ii) of Assumption \ref{ass:NumScheme}, we obtain
\beaa
0 ~\le ~ \T^{t^*,\o^*}_h[ u^h_{t^*+h}] - \T^{t^*,\o^*}_h[\f^\e_{t^*+h} + e^{-Lh}(u^h - \f^\e)(t^*,\o^*) ] +h \bar\d(h).
\eeaa
Since $u^h(t^*,\o^*) = \T^{t^*,\o^*}_h[u^h] $, it follows that
\be\label{final-diff}
\bar\d(h) & \ge & \frac{\T^{t^*,\o^*}_h[\f^\e_{t^*+h} + e^{-Lh}(u^h - \f^\e)(t^*,\o^*)] - u^h(t^*,\o^*)}{h}  \\
& = &   \frac{\T^{t^*,\o^*}_h[\f^\e_{t^*+h} + e^{-Lh}(u^h - \f^\e)(t^*,\o^*)] - \f^\e(t^*,\o^*)+ e^{-Lh}(\f^\e-u^h)(t^*,\o^*)}{h} \notag\\
& & \q\q\q\q\q\q\q\q\q -\frac{(e^{-Lh}-1)(\f^\e-u^h)(t^*,\o^*)}{h}  \notag
\ee
Since $t^*<\bar\ch^{\d,x}(\o^*)$, we have $(t^*,\o^*)\rightarrow 0$ as $\d\rightarrow 0$, and note that
\be\label{uto0}
\lim_{\d,h \rightarrow 0} |u^h(t^*,\o^*) - \ul u _0 | \le \lim_{\d,h \rightarrow 0}\Big( \rho_u\big( d((t^*,\o^*),(0,0))\big) + |u^h_0 - \ul u _0 |\Big) = 0,
\ee
where $\rho_u$ is the common modulus of continuity of the functions $u^h$.
 Further, by \eqref{final-diff} and \eqref{uto0}, it follows from (i) of Assumption \ref{ass:NumScheme}
 \beaa
 0 & \le &- \pa_t \f_0 + \e - G\big(0, \f_0 + \ul u_0 -\f_0,\pa_x \f_0, \pa^2_{xx} \f_0 \big) -L(\f^\e_0- \ul u_0) \notag\\
    & = & - \pa_t \f_0 + \e - G\big(0, \ul u_0 ,\pa_x \f_0, \pa^2_{xx} \f_0 \big) -L\eta.
\eeaa
Finally, we conclude the proof by letting $\e\rightarrow 0$ and then $\eta\rightarrow 0$.

\ms
\no {\bf 2.}\q We now complete the proof of the claim \eqref{obj: step4}. Consider the mixed control and optimal stopping problem in finite discrete-time:
\be \label{eq:YZ_OS}
    Y^h_t (\o) := \inf_{\t\in \cT^h_{\d-t}, \beta \in \Bc}\ul\cE_h [\beta_{\tau} (Z^h)^{t,\o}_\t],~~\mbox{where}~~Z^h_t := (u^h -\f^\e)_t, ~t\in \D_h.
\ee
By standard argument, we have
\beaa
Y^h_0 = \inf_{\beta \in \Bc}\ul\cE_{h} [\beta_{\tau^*}Z^h_{\t^*}],\q\mbox{where}~~\t^* := \inf\{t\in \Delta_h : Y^h_t=Z^h_t \}.
\eeaa
It remains to prove that there exists $\o^*$ such that 
\be\label{objcontradiction}
t^*:=\t^*(\o^*)<\bar\ch^{\d,x}(\o^*)\we (\d-h).
\ee 
Recall that $0 < Z^h \le 2M+\e$ on $[0,\ch^{\d,x}]$ for $h$ small enough. Then since $\t^*\le \d$, we have
\be\label{Zestime}
\ul\cE_h[Z^h_{\t^*}]  & \le &  \ul\cE_h[Z^h_{\t^*};\ch^{\d,x}> \d] +  \ol\cE_h[Z^h_{\t^*};\ch^{\d,x}\le \d] \notag\\
& = &	\inf_{\beta \in \Bc}\ul\cE_h[\beta_{\t^*} Z^h_{\t^*};\ch^{\d,x}> \d] +  \ol\cE_h[Z^h_{\t^*};\ch^{\d,x}\le \d] \notag\\
&\le &   \inf_{\beta \in \Bc}\ul\cE_h[\beta_{\t^*} Z^h_{\t^*}] + \sup_{\b\in\cB}\ol\cE[\beta_{\t^*} |Z^h_{\t^*}|;\ch^{\d,x}\le \d]+  \ol\cE_h[|Z^h_{\t^*} |;\ch^{\d,x}\le \d] \notag\\
&\le &   Y^h_0+ (1+e^{L\d}) \ol\cE_h[|Z^h_{\t^*} |;\ch^{\d,x}\le \d] \notag\\
&\le &   Y^h_0 + (1+e^{L\d})(2M+\e) \ol\cE_h[1_{\{\ch^{\d,x}\le \d\}}].
\ee
On the other hand, we obtain from Lemma \ref{lem: usc conv} that for $h$ small enough it holds
\be\label{eq: stopping time limit}
\ol\cE_h [1_{\{\ch^{\d,x}\le \d\}}] < \ol\cE [1_{\{\ch^{\d,x}\le \d\}}]+\frac{\e\d}{8(4M+2\e)},
\ee
It follows from \eqref{Zestime} and \eqref{eq: stopping time limit} that
\beaa
\ul\cE_h[Z^h_{\t^*}]  ~\le ~ Y^h_0 + (1+e^{L\d})(2M+\e) \ol\cE[1_{\{\ch^{\d,x}\le \d\}}] + \frac{\e\d}{8}.
\eeaa
Further, it follows from \eqref{how small H} that $(1+e^{L\d})(2M+\e)\ol\cE[1_{\{\ch^{\d,x}\le \d\}}] \le \frac{\e\d}{8}$.
Therefore,
\be\label{estime:YZ}
Y^h_0 ~\ge~ \ul\cE_{h} [Z^h_{\t^*}] - \frac{\e\d}{4}.
\ee
Suppose that contrary to \eqref{objcontradiction} we have
\be\label{contradiction os}
\bar\ch^{\d,x}(\o) \we (\d -h) \le \t^*(\o) \le \d, \q\mbox{for all}~~\o.
\ee
Note that
\be\label{eq: Y_0}
\ul\cE_h [Z^h_{\t^*}]  ~\ge  ~ \ul\cE_h[Z^h_\d] + \ul\cE_h [Z^h_{\t^*} - Z^h_{\t^*\vee (\d-h)}] + \ul\cE_h[Z^h_{\t^*\vee (\d-h)} - Z^h_\d].
\ee
It follows from \eqref{how small H} and \eqref{eq: stopping time limit} that
\be\label{i1}
I_1 &:= & \ul\cE_h[Z^h_{\t^*}  - Z^h_{\t^*\vee (\d-h)}]  ~=~	\ul\cE_h[Z^h_{\t^*}  - Z^h_{\d-h}; \t^*<\d-h]\\
& \ge &	-(4M+2\e)\ol\cE_h[1_{\{\ch^{\d,x}\le \d\}}] ~>~ - (4M+2\e)\ol\cE[1_{\{\ch^{\d,x}\le \d\}}]-\frac{\e\d}{8} ~\ge~   -\frac{\e\d}{4}.\notag
\ee
On the other hand, we have
\be\label{eq:I2}
I_2 &:= & \ul\cE_h\big[Z^h_{\t^*\vee (\d-h)} - Z^h_\d\big] \notag\\
& \ge &  - \ol\cE_h\big[(\rho_u+\rho_\f)(h+2\| B_{(\d-h)\we\cd}-B_{\d\we\cd}\|)\big] -\e h,
\ee
where $\rho_u,\rho_\f$ are moduli of continuity of function $u^h$, $\f$, and are chosen to be bounded and continuous. Again by Lemma \ref{lem: usc conv}, we have for $h$ sufficiently small that
\be\label{estime:eh0}
\ol\cE_h \big[  (\rho_u+\rho_\f)(h+2\| B_{(\d-h)\we\cd}-B_{\d\we\cd}\|) \big]
&<& \ol\cE\big[ (\rho_u+\rho_\f)(h+2\| B_{(\d-h)\we\cd}-B_{\d\we\cd}\|) \big] + \frac{\e\d}{8} \notag\\
&=&  \ol\cE\big[ (\rho_u+\rho_\f)(h+2\| B_{h\we\cd} \|\big] + \frac{\e\d}{8},
\ee
It follows from \eqref{eq:I2}, \eqref{estime:eh0} and $\lim_{h\rightarrow 0} \ol\cE\big[ (\rho_u+\rho_\f)(h+2\| B_{h\we\cd} \|\big] =0$ that
\be\label{i2}
I_2 &>& -\frac{\e\d}{4}\q\mbox{for $h$ sufficiently small}.
\ee
Plugging the estimates of \eqref{i1} and \eqref{i2} into \eqref{eq: Y_0}, we obtain
\be\label{ztau-zd}
\ul\cE_h [Z^h_{\t^*}]  ~>  ~   \ul\cE_h[Z^h_\d]  -\frac{\e\d}{2}. 
\ee
Finally, by \eqref{obj: contradiction}, \eqref{estime:YZ} and \eqref{ztau-zd} we have
\beaa
Y^h_0  
~>~ 	 \ul\cE_h[Z^h_\d] - \frac{3\e\d}{4}
~\ge~ 	Z^h_0,
\eeaa
which contradicts the definition of $Y$ in \eqref{eq:YZ_OS}. Therefore, \eqref{contradiction os} is wrong. The proof is complete.
\qed
\end{proof}
	
\section{Conclusion}

	We provide a convergence theorem of monotone numerical schemes for a class of parabolic PPDE,
	which generalizes the classical convergence theorem of Barles and Souganidis \cite{BarlesSouganidis}.
	In contrast to the formulation of \cite{ZhangZhuo}, our conditions are satisfied by all classical monotone numerical scheme for PDEs, to the best of our knowledge.
	Moreover, our results permit to deduce some numerical schemes for path-dependent stochastic differential game problems
	and the second order BSDEs whose generator depends on $z$ (see \eqref{eq:PPDE_2BSDE}, \eqref{eq:PPDE_DG}),
	which are new in literatures.

	Other numerical schemes, such as the branching process scheme of Henry-Labord\`ere, Tan and Touzi \cite{HTT}, are possible for some PPDE, but it is not analyzed by the monotone scheme arguments.

\end{document}